\documentclass[11pt]{article}

\usepackage[mathscr]{eucal}
\usepackage[margin=1in]{geometry}
\usepackage[section]{placeins}
\usepackage{amsrefs}
\usepackage{amsmath}
\usepackage{amsthm}
\usepackage{amsfonts}
\usepackage{amssymb}
\usepackage{hyperref}
\usepackage{enumerate} 
\numberwithin{equation}{section}
\usepackage{enumitem}
\usepackage{comment}
\usepackage{tabularray}
\usepackage{algorithm}
\usepackage{algpseudocode}
\usepackage{subcaption}
\usepackage{graphicx}
\usepackage{bm}

\usepackage{color}
\newcommand{\tbb}{\textcolor{blue}}

\newtheorem{theorem}{Theorem}[section]
\newtheorem{lemma}[theorem]{Lemma}
\newtheorem{proposition}[theorem]{Proposition}
\newtheorem{corollary}[theorem]{Corollary}

\theoremstyle{definition}
\newtheorem{example}[theorem]{Example}
\newtheorem{definition}[theorem]{Definition}
\newtheorem{assumption}[theorem]{Assumption}

\newtheorem{remark}[theorem]{Remark}

\newcommand{\Emb}{{\mathbb{E}}}

\newcommand{\Imb}{{\mathbb{I}}}

\newcommand{\Amc}{{\mathcal{A}}}
\newcommand{\Bmc}{{\mathcal{B}}}

\newcommand{\Dmc}{{\mathcal{D}}}
\newcommand{\Emc}{{\mathcal{E}}}
\newcommand{\Fmc}{{\mathcal{F}}}
\newcommand{\Gmc}{{\mathcal{G}}}
\newcommand{\Hmc}{{\mathcal{H}}}

\newcommand{\Mmc}{{\mathcal{M}}}

\newcommand{\Pmc}{{\mathcal{P}}}

\newcommand{\Smc}{{\mathcal{S}}}
\newcommand{\Tmc}{{\mathcal{T}}}

\newcommand{\Wmc}{{\mathcal{W}}}
\newcommand{\Xmc}{{\mathcal{X}}}
\newcommand{\Ymc}{{\mathcal{Y}}}

\def\R{{\mathbb R}}

\def\N{{\mathbb N}}
\def\PP{{\mathbb P}}

\def\P{{\mathcal P}}

\newcommand{\ps}{\mathscr{P}}

\newcommand{\te}{\text}

\newcommand{\indic}[1]{\Imb_{\left\{#1\right\}}}
\newcommand{\ex}[1]{\Emb\left[#1\right]}
		
\newcommand{\exmu}[2]{\Emb_{#1}\left[#2\right]}	
\newcommand{\deq}{\overset{\text{(d)}}{=}}		
\newcommand{\defeq}{:=}

\newcommand{\borel}{\Bmc}

\newcommand{\ind}{\hspace{24pt}}
\newcommand{\alt}[1]{\widetilde{#1}}
\newcommand{\law}{\te{Law}}

\newcommand{\pcar}[1]{^{(#1)}}

\newcommand{\wsp}{\mathcal{W}}

\newcommand{\mf}{\mathcal{M}}
\newcommand{\incs}{\mathcal{I}}

\allowdisplaybreaks
\title{Large Network limit for Latent Variable Process Models with Endogeneously Evolving Interactions among Multiple Types}

\author{Ankan Ganguly\footnote{Corresponding Author}, and Konstantinos Spiliopoulos\thanks{KS was partially supported by  NSF SES-2120115 and NSF-DMS 2311500.}\\
Department of Mathematics \& Statistics, Boston University\\
665 Commonwealth Avenue, Boston, 02215, MA, United States\footnote{\emph{Email Addresses}: ankan.prob@gmail.com (AG), and kspiliop@bu.edu (KS)}}
\date{}
\begin{document}
\maketitle
\begin{abstract}
Motivated by the opinion and social dynamics literature,  we rigorously study dynamically co-evolving latent space network models, where agent opinions evolve based on a latent state space model influenced by and influencing an underlying dynamically evolving graph determining their interactions.  In contrast to earlier work, see \cite{GanSpiSus25}, agents are allowed to be heterogeneous, i.e., of different types, and their dynamic evolution can capture behaviors like attraction and repulsion between and within groups. In addition, the model encodes bidirectional feedback between the latent process and the underlying interaction network, persistence effects are included, and the interactions are localized, i.e., individual agents do not have global information. We produce weak and empirical limit theorems that encode the asymptotic joint distribution of agent opinion trajectories and their interaction network. In doing so, we describe and exploit the complex dependence structure of the model.

\end{abstract}

\sloppy Keywords: Interacting particle systems, co-evolving networks, opinion dynamics, mean-field, graphons, propagation of chaos


\section{Introduction}
\label{sec::intro}

\ind In many social systems, interactions influence agents' behavior and beliefs, which, in turn, impact interactions. One can find examples of this behavior across the social science and statistics literature, including opinion dynamics~\cite{friedkin1990social,sznajd2000opinion}, 
herding~\cite{banerjee1992simple} or flocking~\cite{douglis1948social,frey2018cognitive}, and polarization~\cite{DelVicario2017,Matakos2017}. 

\ind In the literature of social networks, this relationship between agent behavior and agent interaction is modeled using so-called co-evolving networks. There, the links between the nodes as well as certain attributes of the nodes evolve over time in ways that affect each other. A non-exhaustive list of model approaches within the statistics and social network analysis literature that attempt to conceptualize the dynamic feature of network models include dynamic stochastic block models (e.g., \cite{yang2011detecting,xu2014dynamic,corneli2016exact,matias2017statistical,zhang2017finding}), dynamic latent space network models (e.g., \cite{sarkar2006dynamic, sewell2015latent, sewell2015analysis, sewell2016latent, sewell2017latent,Loyal2023}), temporal exponential family random graph models (e.g., \cite{hanneke2010discrete,krivitsky2014separable}), stochastic actor-oriented models (SAOM, e.g., \cite{snijders1997simulation,snijders2010introduction,snijders2017modeling}), and the recently proposed dynamic co-evolving latent space network with attractor models (CLSNA, \cite{Zhuetal23,Panetal24, CantayDinoHancongEtall2026}).  The SAOM and CLSNA models are inherently co-evolving.

\ind The mathematical community has made tremendous progress in recent years in developing mathematical approaches to study processes on networks (typically coined as interacting particle systems) that have been proposed across the literature. This includes interacting particle systems in applied mathematics and statistical physics  (e.g., \cite{Dawson,GartnerIPS,Garnier2,MFK2008,SirignanoSpiliopoulos2018NN_LLN,SirignanoSpiliopoulos2019_DNN_LLN}), ecology and theoretical biology (e.g., \cite{Dawson,IssacsonMS,SirignanoSpiliopoulos2018NN_LLN}), and economics and game theory (e.g., \cite{AllenGale2000,FrickeLux2015,GlassermanYoung2016,GSS,GSSS,EisenbergNoe2001,Spiliopoulos2015,SpiliopoulosYang2020}) to name just a few. 

\ind  More specifically, and also related to our work, in \cite{BayWu21}, the authors study mean-field limits for a class of continuous-time, discrete-state, interacting particle systems on endogenous, dynamic random graphs; see also \cite{BAYRAKTAR_WU_2024,BayraktarChakrabortyWu2023} for related works. In \cite{BolPao24}, the authors study an opinion dynamics model defined on dynamic, endogenous networks, using a game-theoretic approach and  describe the long-time behavior of the model for finite network sizes. In \cite{AndOlv24}, the authors study mathematically an opinion dynamics model defined on directed graphs closely related to the model of \cite{friedkin1990social}. In a series of papers \cite{LackerRamananWu_2021,LackerRamananWu_2023, LackerRamananWu_2023_2}, the authors study the convergence properties of the empirical measure of systems of interacting particle systems on sparse random graphs, also providing a description of the limiting system for a large set of sparse random graph structures.

\ind This paper considerably builds upon our previous work \cite{GanSpiSus25} in studying the behavior of dynamic co-evolving systems as both the latent process and the underlying graph structure evolve in time and interact with each other (bidirectional feedback effects) through time. \cite{GanSpiSus25} developed the theory needed to comprehensively describe the limiting behavior of such processes by analyzing a simple latent model with purely attractive behavior. It was shown that the mean-field limit of the latent process alone does not fully describe the limit behavior of the process. Instead, a sampling method was used to find mean-field limits, various conditional sampling limits describing joint distributions of particles given their interaction behavior, and graphon limits describing the underlying network. In this work, we study a notably more complex and more realistic model that is capable of handling heterogeneous particle and interaction types and a greater variety of interaction rules between interacting particles. The infrastructure we develop supports a rich collection of latent process behaviors which also demonstrate a wide variety of different behaviors. The richer dynamical structure of the models studied in this paper yields a number of additional mathematical complexities which we resolve in this paper. Details of the mathematical advances are discussed later on.

\ind As in \cite{GanSpiSus25}, we study the behavior of such systems when they are subject to persistence effects, meaning that the graph structure at the next time instance does not depend only on the corresponding latent process values but also on the graph structure at the current time instance. Essentially, this mirrors the idea that if an edge is present now, it may be more likely to be present in the next time instance as well. As a running example, we keep in mind the CLSNA model class \cite{Zhuetal23} which has all of these characteristics present.

\ind Let us next present in greater detail the model we study. Consider a model for which at each time $t \in \N_0$, we model $n$ agents holding latent opinions represented by the set of vectors $Z^n(t) \defeq (Z^n_i(t))_{i=1}^n \in \R^{d\times n}$. Agents are equipped with dynamic types (denoted by $Y^n_i$(t)) lying in the set $[1:q]$. Interactions between agents at time $t$ are encoded by a dynamic (undirected) network with adjacency matrix $A^n(t) \defeq (A^n_{ij}(t))_{i,j=1}^n$ with $m$ different edge types. For example, if $A^n_{ij}(t) = a$, then the nodes $i$ and $j$ have an edge of type $a$ between them and if $A^n_{ij}(t) = 0$, then there is no edge between the nodes $i$ and $j$. At time $t+1$, agent $i$ updates their opinion by considering a type-dependent linear combination of their own opinion and the average opinion of their neighbors (along each edge type) at time $t$, then perturbing that opinion with i.i.d. additive noise $\{\xi_i(t)\}_{i\in \N,t\in\N_0}$:
\begin{align}
\label{eqn::Znevolve}
Z^n_{i}(t+1) &= \Lambda_0(Y^n_i(t))Z^n_{i}(t) + \sum_{k=1}^m \Lambda_k(Y^n_i(t)) L^{n,k}_i(t) + \xi_i(t),\\
\label{eqn::Lnidef}
L^{n,k}_i(t) &= \frac{1}{d_{A^{n,k}(t)}(i)}\sum_{j=1}^n Z^n_j(t)\indic{A^n_{ij}(t)=k},\\
\label{eqn::dAnt}
d_{A^{n,k}(t)}(i) &= \sum_{j=1}^n \indic{A^n_{ij}(t)=k},
\end{align}
where $\Lambda_0,\dots,\Lambda_m$ are measurable mappings from $[1:q]$ to $\R^{d\times d}$ represent random matrices, $A^{n,k}(t)$ is the adjacency matrix defined by $A^{n,k}_{ij}(t) = \indic{A^n_{ij}(t) = k}$, and where we set $L^{n,k}_i(t) = 0$ in the event that the degree of $i$ in $A^{n,k}(t)$ given by $d_{A^{n,k}(t)}(i)$ is equal to 0.

\ind Let $\Fmc^{Y,n}_t \defeq \sigma(Z^n[t],Y^n[t-1],A^n[t-1])$ and $\Fmc^{A,n}_t\defeq \sigma(Z^n[t],Y^n[t],A^n[t-1])$. Then for each $t \in \N$, we assume $\{Y^n_i(t)\}_{i \in \N}$ are conditionally independent given $\Fmc^{Y,n}_t$. Furthermore, there exists an a.e. continuous function $P: \R^d\times [1:q]^2 \to [0,1]$ such that for each $y \in [1:q]$,
\begin{equation}
\label{eqn::peqn}
\PP\left(Y^n_i(t+1) = y\middle|\Fmc^{Y,n}_t\right) = P(Z^n_i(t+1),Y^n_i(t),y).
\end{equation}

\ind The network $A^n$ also evolves with time. Fix any $t \in \N_0$ and $1\leq i < j \leq n$. At time $t+1$, there exists an edge between agents $i$ and $j$ with probability depending on \textbf{(a)} their latent positions $Z^n_{\{i,j\}}(t+1) \defeq (Z^n_i(t+1),Z^n_j(t+1))$, \textbf{(b)} their node types $Y^n_{\{i,j\}}(t+1) \defeq (Y^n_i(t+1),Y^n_j(t+1))$ and \textbf{(c)} the interaction between the agents $A^n_{ij}(t)$ at time $t$. This probability is determined by an a.e. continuous \emph{interaction function} $B: (\R^d)^2\times [1:q]^2 \times [m]^2\to [0,1]$. That is, for all $a \in [m]$,
\begin{align}
\PP\left(A^n_{ij}(t+1) = a\middle|\Fmc^{A,n}_t\right) &= \PP\left(A^n_{ji}(t+1) = a\middle|\Fmc^{A,n}_t\right)\nonumber\\
& = B\left(Z^n_{\{i,j\}}(t+1),Y^n_{\{i,j\}}(t+1),A^n_{ij}(t),a\right).\label{eqn::Anevolve}
\end{align}
In summary, \eqref{eqn::Anevolve} implies that the only relevant factor determining the type of interaction between two agents at time $t+1$ is the latent positions of the agents at time $t+1$, the types of the agents at time $t+1$ and the network at time $t$. To this end, it is assumed that $\{A^n_{ij}(t+1)\}_{1\leq i < j\leq n}$ are mutually conditionally independent given $\Fmc^{A,n}_t$.

\ind In this work, we study the limit of this dynamical system as it grows bigger in size, i.e., as $n\rightarrow\infty$. We obtain a mean-field convergence result for the latent particle system, Theorem \ref{thm::hydro}, and the convergence of the latent particle network (graphon), Theorem  \ref{thm::Antmultconv} and Corollary \ref{coro::Wdefthm}. To accomplish these goals, we first derive the convergence of the joint distribution of the pair of sample latent positions and the associated network trajectory, explicitly characterizing the limit point, Theorem \ref{thm::Sampleconv}. Understanding and properly analyzing the dependence and correlation structure in this class of models is one of the key technical difficulties of this work.

\ind When $m = q = 1$, $\Lambda_0(\cdot) \equiv\gamma I_d$ for some $0 \leq \gamma \leq 1$ and $\Lambda_1(\cdot) \equiv I_d-\Lambda_0$, the model introduced above is nearly equivalent to the coevolving latent variable process introduced in \cite{GanSpiSus25} with only a few minor changes that vanish in the limit (see Remark \ref{rem::iiis0}). Our work considerably extends the results of that paper to allow for a larger variety of population dynamics models. First, our model introduces heterogeneity in the sense that particles can be sorted into different dynamic types (this may be used to model, for example, political affiliation). In addition, we introduce different dynamic interaction types whose evolution depends on a particle's latent opinions and its type. For example, we may allow some particles to be attracted to each other while others are repelled from each other. Our use of matrices $\Lambda_k$, $k \in [1:m]$ in \eqref{eqn::Znevolve} allow for other, more complex interactions. For instance, if $d = 2$, we may use the latent space to model latent opinions on two political issues. For instance, our model allows us to consider the case when interactions may bring voters closer together on one issue and separate voters in another. 

\ind We built on the results of \cite{GanSpiSus25} to analyze the general multi-type model (\ref{eqn::Znevolve})-(\ref{eqn::Anevolve}). This extension introduces several interesting technical challenges, and it is practically more relevant as well. In both papers, the model exhibits a strong conditional structure which is leveraged to develop law of large numbers type results. However, the introduction of dynamic particle types significantly complicates this conditional structure and the resulting arguments relying on such structure. Likewise, the arguments of both papers require $\{Z^n_i(t)\}_{i \in \N, n \in \N}$ to be uniformly integrable random vectors. In \cite{GanSpiSus25}, we were able to find relatively simple sufficient conditions for this to hold. The analogous sufficient conditions for our model are complex enough to be relegated to a section of their own (Appendix \ref{sec::ui}). 

\ind Of particular note, establishing the limit of the ratio term(s) is significantly more difficult in this work compared to \cite{GanSpiSus25}. In both papers, the ratio term is decomposed as follows:
\begin{equation}
\label{Lnpf}
L^{n,k}_i(t) = \frac{\sum_{j=1}^n Z^n_j(t)\indic{A^n_{ij}(t)=k}}{d_{A^{n,k}(t)}(t)} = \frac{\frac{1}{n}\sum_{j=1}^n Z^n_j(t)\indic{A^n_{ij}(t)=k}}{\frac{1}{n}\sum_{j=1}^n \indic{A^n_{ij}(t)=k}},
\end{equation}
where we then examine the limits of both the numerator and denominator. In \cite{GanSpiSus25}, it is assumed that there is only one interaction type, so in such models we naturally show that each particle $j$ interacts with $\theta(n)$ other particles. This shows that the denominator above has a non-zero limit, so we can simply apply a conditional law of large numbers type argument (given $Z^n_i(t)$) and the continuous mapping theorem to find the limiting behavior of the ratio term $L^{n,k}_i$. In this paper, the argument is not as simple. As there are multiple particle types and interaction types, our model must account for the possibility that a given particle $i$ has type $Y^n_i = y$ and that particles of type $y$ do not experience interactions of type $k$. In this case, the denominator above \emph{would} be zero. Another possibility we must account for is the case in which a given particle $i$ may interact with $o(n)$ other particles with probability not converging to zero. Handling these cases requires careful analysis of the respective supports of the $n$-particle system and the mean-field process. We examine these cases in detail in Section \ref{ssec::tech}, where we show that a particle $i$ either has $\theta(n)$ neighbors of type $k$ or no neighbors of type $k$. Moreover, we show that these events have null boundaries in the limit, which allows us to apply the continuous mapping theorem.

\ind The bidirectional interactions between the network and agent opinions result from our choice to allow $B$ from \eqref{eqn::Anevolve} to depend on the state of the agents $Y^n_{\{i,j\}}$. The dependence of $B$ on the network $A^n_{ij}$ at the previous time instance induces the so-called persistence effect in the social science/opinion dynamics literature (see, for example, \cite{Zhuetal23}). Subsequently, this leads to an intriguing dependence structure in the limiting model. We provide a thorough examination of the probabilistically rich conditional structure of the limiting model. This is achieved in Proposition \ref{prop::welldef} and Corollary \ref{coro::morestruct}. These results characterize the limiting model fully. In order to characterize this conditional structure, 
we use the conditional propagation of chaos result proved in \cite{GanSpiSus25} and recalled here as Proposition \ref{prop::condconvres} in Section \ref{sec::condpropchaos}. In addition, in order to derive the limiting behavior, we apply the sample perspective idea of  \cite{GanSpiSus25}, appropriately modified to handle the intricacies of the model (\ref{eqn::Znevolve})-(\ref{eqn::Anevolve}).

\ind To relate more concretely to the related statistics and opinion dynamics literature, the model (\ref{eqn::Znevolve})-(\ref{eqn::Anevolve}) covers as special cases some of the applied network models in the literature. In particular, we mention the following examples.

\begin{example}[BCOD Model]
\label{ex::BCOD}

When the interaction kernel $B$ takes the form
\[ B(z_{\{1,2\}},y_{\{1,2\}},b,a) = \indic{\|z_1-z_2\|_2 \leq \delta},\]
then our model becomes an \emph{opinion dynamics with bounded confidence} (BCOD) model, e.g. \cite{HegKra02}. Specifically, in the model of \cite{HegKra02}, the authors take $\Lambda_0(y) = 1$, $m=1$,  $\Lambda_1(y) = 0$  and the noise component $\xi_i(t)=0$ for all agents $i$.

\end{example}

\begin{example}[Dynamic latent space models without coevolution] 
In the typical dynamic latent space network models, see for example \cite{sarkar2006dynamic,Sewell02102015}  latent characteristics
evolve in time in a Markov manner, and the links between nodes are driven by probabilities
determined by node distance in the underlying latent space; that is, the interaction function $B$ does not depend on the network at the previous time. This class of models includes the BCOD model of Example \ref{ex::BCOD}. In particular, those models build on a base model in which $q=m=1$, $\Lambda_0 \equiv 1$,  $\Lambda_1 \equiv 0$, and $\xi_i(t) \overset{\te{i.i.d.}}{\sim} \mathcal{N}(0,\tau^2I_d)$, where $\tau$ is a parameter. The interaction kernel $B$ takes the form
\begin{align*}
\te{logit}(B(z_{\{1,2\}},y_{\{1,2\}},b,a)) = \alpha -\|z_1-z_2\|,
\end{align*}
where $\alpha$ is a parameter. It is easily verified that this model satisfies all of the assumptions of Section \ref{sssec::assu}.
\end{example}

\begin{example}[CLSNA Model, \cite{Zhuetal23,Panetal24, CantayDinoHancongEtall2026}]
\label{ex::CLSNA}
Consider the following model: we set $\xi_i(t) \overset{\te{i.i.d.}}{\sim} \mathcal{N}(0,\tau^2I_d)$, where $\tau$ is a parameter. We assume that all particles are separated into two static types, $y=1$ and $y=2$, defined by the static type evolution function
\[P(z,y_0,y_1) = \indic{y_0 = y_1}.\]
The model has three interaction types: for $y = 1,2,$ type $y$ interactions are attractive interactions between two type $y$ particles. Type 3 interactions are repulsive interactions between particles of different types. The strength of attraction/repulsion is governed by parameters $\gamma^w_y > 0$, $y = 1,2,$ and $\gamma^b < 0$. Under this model, $\Lambda_0(y) = (1 - \gamma^w_y-\gamma^b)I_d$, $\Lambda_1(y) = \gamma^w_yI_d$ and $\Lambda_2(y) = \gamma^bI_d$. The interaction kernel $B$ takes the form
\begin{align*}
\te{logit}(B(z_{{1,2}},y_{{1,2}},b,a)) = \begin{cases}
\alpha + \delta\indic{a=b}-|z_1-z_2| &\te{ if } y_1=y_2\te{ and }a = y_1\te{ or } y_1\neq y_2\te{ and } a = 3,\\
-\infty &\te{ otherwise,}
\end{cases}
\end{align*}
where $\alpha$ and $\delta$ are parameters. Using the fact that particle types are assumed to be static, the results of this paper hold for this model. The uniform integrability of this model (Assumption \ref{assu::ui}) is established in Proposition \ref{prop:uisuffconds}. Specifically, it is  shown that the model satisfies Assumptions \ref{assu::ybds} and \ref{assu::bdsy}.
\end{example}

\ind The rest of the paper is organized as follows. In Section \ref{ssec::assures}, we  state our results more precisely. Specifically, in \ref{sssec::assu}, we describe the assumptions under which our main results hold. In Section \ref{sssec::samplelimit}, we describe the mean-field limit in detail and state our main result: the convergence of a random sample of $k$ particles to the mean-field limit. In Section \ref{ssec::Gconv}, we describe the graphon convergence of the patent particle network process. In Section \ref{sec::proofs}, we prove our results. The paper comes with three appendices. In Appendix \ref{sec::ui} we establish sufficient conditions for uniform integrability, i.e., conditions that lead to Assumption \ref{assu::ui}. In Appendix \ref{sec::fnzsjust}, we present a technical result which is used in the proofs. In Section \ref{sec::condpropchaos}, we summarize a general conditional propagation of chaos result introduced in \cite{GanSpiSus25} which is fundamental to our proof.

\subsection{Notation}
\label{ssec::notation}

Throughout, the sets $\Xmc$ and $\Ymc$ are always assumed to be arbitrary Polish spaces (such as Euclidean space). Same holds for $\Xmc_i$ or $\Ymc_i$ for $i \in I$, where $I$ is any index set. $\P(\Xmc)$ is the space of Borel probability measures on $\Xmc$ equipped with the topology of weak convergence. $C_b(\Xmc)$ is the space of bounded, continuous functions from $\Xmc$ to $\R$. $C_0(\Xmc)$ is the space of continuous functions vanishing at infinity (the uniform closure of the set of continuous functions with compact support). For any $U \subseteq \R$, $C_b(\Xmc,U)$ is the space of bounded, continuous functions from $\Xmc$ to $U$. Both $C_0(\cdot)$ and $C_b(\cdot)$ are equipped with the uniform topology. Lastly, we use $\ps(\Xmc)$ to denote the power set of $\Xmc$; this is only used for certain finite $\Xmc$ for which any measurability concerns are trivially satisfied.

\ind For integers $a < b$, let $[a:b] = \{a,a+1,\dots,b\}$ be the sequence of integers from $a$ to $b$. If $a = 0$, then we simply write $[b]\defeq [0:b]$. This notation is also used in vertex/matrix indices. For example, $x_{1:n} = (x_1,\dots,x_n)$ and $M_{1:k,1:k} = (M_{ij})_{i,j \in 1:k}$. Let $\N = \{1,2,\dots\}$ denote the natural numbers and let $\N_0 = \{0\}\cup \N$ be the set of whole numbers. For $a,b \in \R$, $a\wedge b = \min\{a,b\}$ and $a \vee b = \max\{a,b\}$. Given $x \in \R^d$, $|x| = \sqrt{x\cdot x}$. Given a function $f: \Xmc \to \R^d$, $\|f\|_{\infty} = \sup_{x \in \Xmc}|f(x)|$ is the uniform norm of $f$. $S_n$ is the set of permutations of the set $[1:n]$.

\ind Let $\mf(\Xmc)$ be the space of functions from $\N_0$ to $\Xmc$ and let $\mf_t(\Xmc)$ be the space of functions from $[t]$ to $\Xmc$. Given a process $x\in \mf(\Xmc)$, we write $x(t)$ for the value of $x$ at time $t$ and $x[a:b] \defeq (x(a),\dots,x(b))$ for the trajectory of $x$ in the time interval $[a:b]$. For $t \in \N_0$, we shorten $x[0:t]$ to $x[t]\in \mf_t(\Xmc)$. 

\ind For $k \in \N$, we use the notation $\incs_k$ for the subset $\{(i,j): 1\leq i < j \leq k\}\subset [1:k]^2$. Finally, throughout the article, we define $0/0 = 0$.

\section{Assumptions and Limiting Behavior of the Latent Particle System}
\label{ssec::assures}

In this section, we outline the assumptions we make about the model as well as the asymptotic results we derive on the convergence of the latent process. The assumptions of Section \ref{sssec::assu} are similar in form to the assumptions of \cite{GanSpiSus25}, but modified to allow for multiple particle types and interaction types. 

\subsection{Assumptions}
\label{sssec::assu}

We begin by explicitly stating a few assumptions about the conditional structure of the process described in \eqref{eqn::Znevolve}-\eqref{eqn::Anevolve}. For each $n \in \N$ and $t \in \N$, recall that $\Fmc^{Y,n}_t \defeq \sigma(Z^n[t],Y^n[t-1],A^n[t-1])$, and $\Fmc^{A,n}_t \defeq \sigma(Z^n[t],Y^n[t],A^n[t-1])$. Define $\Fmc^{Y,n}_0 \defeq \sigma(Z^n(0))$ and $\Fmc^{A,n}_0 \defeq \sigma(Z^n(0),Y^n(0))$. Moreover, for $t \geq 0$, define $\Fmc^n_t \defeq \sigma(Z^n[t],Y^n[t],A^n[t])$.

\begin{assumption}[Conditional Structure and Absolute Continuity]
\label{assu::cond}
The process $(Z^n,A^n,\xi)$ possesses the following properties:
\begin{enumerate}[label = (\alph*)]
\item The collection $\{\xi_i(t)\}_{i \in \N,t\in \N_0}$ is i.i.d., absolutely continuous with full support and for each $t \in \N_0$, $(\xi_i(t))_{i\in\N}$ is independent of $\bigvee_{n \in \N}\Fmc^n_t$.

\item For each $n \in \N$ and $t \in \N_0$, $\{Y^n_i(t)\}_{i \in [1:n]}$ are mutually independent $[m]$-valued random variables given $\Fmc^{Y,n}_t$.

\item For each $n\in \N$ and $t \in \N_0$, $A^n(t)$ is a symmetric $[m]$-random matrix where $\{A^n_{ij}(t)\}_{(i,j) \in \incs_n}$ is a conditionally independent collection of $[m]$-valued random variables given $\Fmc^{A,n}_{t}$. In addition, for any $i\leq n$, $A^n_{ii}(t) = 0$.
\end{enumerate}
\end{assumption}

\begin{remark}
\label{rem::iiis0}
In \cite{GanSpiSus25}, it was assumed that for any $i\leq n$, $A^n_{ii}(t) = 1$. This was necessary to ensure that the denominator of the averaging term $L^n$ would always be non-zero. In the more general setting studied in this paper, we have multiple interaction types, so it no longer makes sense to privilege any specific interaction type. This is why we set $A^n_{ii}(t) = 0$. To ensure that \eqref{eqn::Lnidef} is always well-defined, we apply the convention $0/0 = 0$. This distinction from the model in \cite{GanSpiSus25} vanishes in the limit as $n \to \infty$.
\end{remark}

We assume the initial distribution of the process satisfies the following condition:

\begin{assumption}[Initial Conditions]
\label{assu::init}
For all $n \in \N$, the initial state $(Z^n(0),Y^n(0),A^n(0))$ satisfies the following conditions:
\begin{enumerate}[label = (\alph*)]
\item $Z^n_{1:n}(0)$ are absolutely continuous with full support, exchangeable and there exists an i.i.d. absolutely continuous $\R^d\times \R^d$-random vector pair $Z\pcar{1,2}(0)$ with full support and distribution $\mu_0\times \mu_0 \in \P(\R^d\times \R^d)$ such that convergence in distribution holds:
\[Z^n_{1,2}(0) \Rightarrow Z\pcar{1,2}(0).\]
\item There exists an a.e. continuous function $P_0: \R^d \times [1:q] \to [0,1]$ such that for each $i \in [1:n]$ and $y \in [1:q]$,
\[\PP\left(Y^n_i(0) = y\middle|\Fmc^{Y,n}_0\right) = P_0\left(Z^n_i(0),y\right).\]
Moreover, $\{Y^n_i(0)\}_{i=1}^n$ are mutually conditionally independent given $\Fmc^{Y,n}_0$.
\item there exists an a.e. continuous function $B_0: (\R^d)^2\times [1:q]^2\times [m] \to [0,1]$ such that for all $(i,j) \in \incs_n$ and $a \in [m]$,
\[\PP\left(A^n_{ij}(0) = a\middle|\Fmc^{A,n}_0\right) = B_0\left(Z^n_{\{i,j\}}(0),Y^n_{\{i,j\}}(0),a\right).\]
\end{enumerate} 
\end{assumption}

Finally, we require uniform integrability:

\begin{assumption}[Uniform Integrability]
\label{assu::ui}
For every $t \geq 0$, the collection of all latent positions $\{Z^n_i(t)\}_{i\in [1:n],n \in \N}$ is uniformly integrable.
\end{assumption}

Establishing the uniform integrability of the process can be difficult and/or tedious in practice, so in Appendix \ref{sec::ui}, we provide two easily verifiable sufficient conditions under which the latent positions of the process are uniformly integrable. Both sufficient conditions extend Assumption 2.3 of \cite{GanSpiSus25} to our more complex model. They both require exponential moment bounds on the initial latent state $Z^n_1(0)$ and the driving noise $\xi_1(0)$.

\ind The first sufficient condition (Assumption \ref{assu::bdsgen}) holds for models with dynamic particle types. For each type of interaction $k$, one of the following must hold:
\begin{enumerate}
    \item The probability of an edge of type $k$ existing between two particles decreases at most exponentially in the distance between their latent positions.
    \item Particles can only experience interactions of type $k$ with other particles within some finite range $C_b$.
\end{enumerate}

Our second sufficient condition (Assumptions \ref{assu::ybds} and \ref{assu::bdsy}) holds in the special case that particle types are static. Under this condition, it is possible for pairs of particles to satisfy different bounds (exponential vs finite range) depending on their type. This has applications in models such as the CLSNA model described in Example \ref{ex::CLSNA} in which, depending on their respective types, particle interaction probabilities either decrease exponentially in the distance of their latent positions, or particles do not interact.

Lastly, we add a final minor assumption on the functions $B_0,P_0,B$ and $P$.

\begin{assumption}
\label{assu::Bbds}
For each $(y,y'_0,y'_1,a,a') \in [1:q]^3\times[m]^2$, and $0\leq s \leq t$, define
\begin{align*}
\Dmc_0(y,y'_0,a) &\defeq \{z\in \R^d: B_0(z,z',y,y'_0,a)P_0(z',y'_0) = 0\te{ for a.e. }z' \in \R^d\},\\
\Dmc(y,y'_0,y'_1,a,a') &\defeq \{z\in \R^d: B(z,z',y,y'_1,a,a')P(z',y'_0,y'_1) = 0\te{ for a.e. }z' \in \R^d\}.
\end{align*}
Then for every $(y,y'_0,y'_1,a,a') \in [1:q]^2\times[m]^2$, $\partial \Dmc_0(y,y'_0,a)$ and $\partial \Dmc(y,y'_0,y'_1,a,b)$ are both of Lebesgue measure 0. 
\end{assumption}

\begin{remark}[Assumption \ref{assu::Bbds} Application]
\label{rem::newassu}
Assumption \ref{assu::Bbds} is a technical assumption. For it to be violated, $B_0$, $P_0$, $B$ and/or $P$ would have to be functions for which the boundaries of their supports would be of non-zero Lebesgue measure. There are pathological examples of such functions. However, we do not expect such functions to arise in any standard applications of our framework. The assumption is needed in models for which certain types of particles may have no interactions of a given type $k \in [m]$. In these models, we expect for each particle to have either $0$ interactions of a given type $k$ or $\Theta(n)$ interactions of type $k$ which is crucial for our approach to establishing the mean-field limit. Under Assumption \ref{assu::Bbds}, this intuition holds true.
\end{remark}

\subsection{Limiting Model over a Randomly Sampled Subnetwork}
\label{sssec::samplelimit}

In this section, we describe the behavior as $n\to\infty$ of the latent opinions particle types, and interaction subnetwork generated by a random sample (with or without replacement) of $k$ agents. In the limit, the agents have i.i.d. latent opinions $Z\defeq Z\pcar{1:k}$ satisfying equations \eqref{eqn::Zevolve}-\eqref{eqn::mu2def} below. The limiting subnetwork connecting the agents given by the Bernoulli random variables $A\defeq A\pcar{1:k,1:k}$ with a rich conditional structure detailed in Definition \ref{defn::AZlim} of Section \ref{ssssec::limit}. In Section \ref{ssssec::intuition}, we provide simple heuristic arguments that motivate our description of the limiting behavior of the $k$-agent sample. In Section \ref{ssssec::limit}, we describe the limiting model $(Z\pcar{1:k}[t],A\pcar{1:k,1:k}[t])$ in detail. In Section \ref{ssssec::cstruct} we discuss the conditional structure of the limiting model. Then, in Section \ref{ssssec::mainres}, we rigorously state our main convergence result: that the limiting distribution of the latent opinions and subnetwork generated by a sample of size $k$ is given by $(Z\pcar{1:k}[t],A\pcar{1:k,1:k}[t])$. Our presentation follows in parallel to the analogous derivation in Section 2.2 of \cite{GanSpiSus25} while emphasizing the adjustments needed to accommodate the richer model studied in this paper.

\subsubsection{ Building Intuition}
\label{ssssec::intuition}

Suppose that $n$ is large, and we select a sample of $r$ agents uniformly at random with or without replacement. By a simple exchangeability argument, it can be shown that for all $n \in \N$, the joint opinion dynamics and subnetwork connecting nearly any collection of $r$ agents have the same distribution (with deviations only when agents are chosen multiple times within the same sample). We may, therefore, assume without loss of generality that our sample consists of agents $1,\dots,r$. To keep things simple, let's assume the latent opinions and the subnetwork converge weakly:
\[(Z^n_{1:r}[t],Y^n_{1:r},A^n_{1:r,1:r}[t]) \Rightarrow (Z\pcar{1:r}[t],Y\pcar{1:r}[t],A\pcar{1:r,1:r}[t]).\]
What can we say about the limiting quantities?

\ind First, our assumptions ensure that for large $n$, each agent interacts with at least $cn$ other agents, where $c > 0$ is some (possibly random) constant. This suggests that interactions between any finite set of agents vanish in the limit as $n\to\infty$. In addition, we can infer that the limiting opinions and particle types $\{(Z\pcar{i}[t],Y\pcar{i}[t])\}_{i \in [1:r]}$ are independent. Exchangeability also suggests that they are identically distributed.

\ind The dynamic subnetwork described by $A[t]$ is much more interesting due to its rich conditional structure. Just as in the $n$ particle case, at time $t > 0$, we can construct a new subnetwork $A(t)$ at each time $t$ by setting $A\pcar{ij}(t) = a$ with probability $B(Z\pcar{i,j}(t),Y\pcar{i,j}(t),A\pcar{ij}(t-1),a)$ for each $a \in [m]$ and $(i,j) \in \incs_r$. Moreover, just as in the $n$-particle case, the edges $\{A\pcar{ij}(t)\}_{(i,j)\in \incs_r}$ are conditionally independent given $\sigma(A[t-1],Y[t],Z[t])$. However, in the $n$-particle case, $Z_i[t]$ and $Z_j[t]$ are not independent, while $Z\pcar{i}[t]$ and $Z\pcar{j}[t]$ are independent. This induces a rich conditional structure on $A[t]$. For example, we show the following.
\begin{itemize}
\item For $s > 0$, the random variables $\{A\pcar{ij}(s)\}_{(i,j)\in \incs_r}$ are conditionally independent given $\sigma(A(s-1),Y(s),Z(s))$.
\item The random elements $\{A\pcar{ij}[t]\}_{(i,j)\in \incs_r}$ are conditionally independent given $\sigma(Z[t],Y[t])$.
\item For each $i\neq j$, $A\pcar{ij}[t]$ is a conditional Markov chain given $\sigma(Z[t],Y[t])$. Its Markov kernel is given by the function $B$.
\item The \emph{edge-process} $A\pcar{ij}[s]$ is conditionally independent of $\sigma(Z[s+1:t],Y[s+1:t])$ given $\sigma(Z[s],Y[s])$.
\end{itemize}

We can make use of the conditional Markov chain formulation to find the conditional probability that $A\pcar{ij}(s) = a$ given $\sigma(Z\pcar{i,j}[t],Y\pcar{i,j}[t])$. Describe this probability by the function $B_{s,a}: \mf_t(\R^d)^2\times \mf_t([1:q])^2 \to \R$ satisfying
\begin{equation}
\label{eqn::Bsdef}
B_{s,a}(Z\pcar{i,j}[t],Y\pcar{i,j}[t]) = \PP\left(A\pcar{ij}(s) = a\middle|Z\pcar{i,j}[t],Y\pcar{i,j}[t]\right).
\end{equation}
This conditional structure and the functions $\{B_{s,a}\}_{s \in \N_0,a \in [m]}$ can be leveraged to derive the dynamics of the limiting opinions $Z[t]$ and particle types $Y[t]$. To do this, we show that for any $i\in \N$ and as $n \to \infty$, the sequence of paired random elements $\{(Z^n_j(t),Y^n_j(t),A^n_{ij}(t))\}_{j=1}^n$ becomes approximately conditionally independent given $(Z^n_i[t],Y^n_i[t])$. Therefore, we can prove the following conditional law of large numbers results:
\begin{align}
\label{eqn::cllna}
\frac{1}{n}\sum_{j=1}^n Z^n_j(t)\indic{A^n_{ij}(t)=k} &\Rightarrow \ex{Z'(t)\indic{A'(t)=k}\middle|Z\pcar{i}(t),Y\pcar{i}(t)}\\
\frac{1}{n}\sum_{j=1}^n A^n_{ij}(t) &\Rightarrow \ex{A'(t)\middle|Z\pcar{i}(t),Y\pcar{i}(t)},\label{eqn::cllnb}
\end{align}
where $(Z\pcar{i}[t],Z'[t],Y\pcar{i}[t],Y'[t],A'[t])\deq (Z\pcar{1,2}[t],Y\pcar{1,2}[t],A\pcar{12}[t])$. With this, we can derive the limit of the term $L^{n,k}_i(t)$ from \eqref{eqn::Lnidef}:

\begin{align}
L^{n,k}_i(t) &= \frac{1}{d_{A^{n,k}(t)}(i)}\sum_{j=1}^n Z^n_j(t)\indic{A^n_{ij}(t)=k}\nonumber\\
&= \frac{\frac{1}{n}\sum_{j=1}^n Z^n_j(t)\indic{A^n_{ij}(t)=k}}{\frac{1}{n}\sum_{j=1}^n \indic{A^n_{ij}(t)=k}}\nonumber\\
&\Rightarrow \frac{\ex{Z'(t)\indic{A'(t)=k}\middle|Z\pcar{i}[t],Y\pcar{i}[t]}}{\PP\left(A'(t)=k\middle|Z\pcar{i}[t],Y\pcar{i}[t]\right)}\nonumber\\
&= \frac{\ex{Z'(t)B_{t,k}(Z\pcar{i}[t],Z'[t],Y\pcar{i}[t],Y'[t])\middle|Z\pcar{i}[t],Y\pcar{i}[t]}}{\ex{B_{t,k}(Z\pcar{i}[t],Z'[t],Y\pcar{i}[t],Y'[t])\middle|Z\pcar{i}[t],Y\pcar{i}[t]}}\nonumber\\
&\defeq L\pcar{i}_k(t).\label{eqn::Liintuit}
\end{align}

From here, we can then derive the following limiting dynamics for $Z\pcar{i}$:
\begin{equation}
\label{eqn::Zintuit}
Z\pcar{i}(t+1) = \Lambda_0Z\pcar{i}(t) + \sum_{k=1}^m \Lambda_k L\pcar{i}_k(t) + \xi_i(t).
\end{equation}

\subsubsection{A Description of the Limiting Model}
\label{ssssec::limit}

In this section, we provide a complete description of the distribution of the limiting latent opinions and subnetwork of a random sample up to some arbitrarily given time $t \in \N_0$: $(Z\pcar{1:r}[t],Y\pcar{1:r}[t],A\pcar{1:r,1:r}[t])$. Importantly, we provide multiple characterizations of the conditional structure of the model. For this, we make use of the following filtrations. For each $0\leq s \leq t$ define 
\begin{equation}
\label{eqn::filtrations}
\Tmc^r_s \defeq \sigma\left(Z\pcar{1:r}[s],Y\pcar{1:r}[s],A\pcar{1:r,1:r}[s]\right),\quad \Gmc^r_s \defeq \sigma\left(Z\pcar{1:r}[s],Y\pcar{1:r}[s]\right)\te{ and } \Gmc^{Z,r}_s \defeq \sigma\left(Z\pcar{1:r}[s]\right),
\nonumber
\end{equation}
\begin{align*}
\Gmc^{A,r}_t &\defeq \begin{cases}
\sigma\left(Z\pcar{1:r}[s],Y\pcar{1:r}[s],A\pcar{1:r,1:r}[s-1]\right)&\te{ if } s > 0,\\
\sigma\left(Z\pcar{1:r}(0),Y\pcar{1:r}(0)\right)&\te{ if } s = 0.
\end{cases} \\
\Gmc^{Y,r}_t &\defeq \begin{cases}
\sigma\left(Z\pcar{1:r}[s],Y\pcar{1:r}[s-1],A\pcar{1:r,1:r}[s-1]\right)&\te{ if } s > 0,\\
\sigma\left(Z\pcar{1:r}(0)\right)&\te{ if } s = 0.
\end{cases}
\end{align*}

For $0\leq s \leq t$, we also define the $\sigma$-algebras $\Hmc^{Y,r}_s$ and $\Hmc^{A,r}_s$ by
\begin{align}
\label{eqn::HmcY}
\Hmc^{Y,r}_s &\defeq \begin{cases}
\sigma\left(Z\pcar{1:r}(s),Y\pcar{1:r}(s-1),A\pcar{1:r,1:r}(s-1)\right) &\te{ if } s > 0,\\
\sigma\left(Z\pcar{1:r}(0)\right) &\te{ if } s= 0.
\end{cases}\\
\Hmc^{A,r}_s &\defeq \begin{cases}
\sigma\left(Z\pcar{1:r}(s),Y\pcar{1:r}(s),A\pcar{1:r,1:r}(s-1)\right) &\te{ if } s > 0,\\
\sigma\left(Z\pcar{1:r}(0),Y\pcar{1:r}(0)\right) &\te{ if } s= 0.
\end{cases}\label{eqn::HmcA}
\end{align}

Combining \eqref{eqn::Bsdef}, \eqref{eqn::Liintuit} and \eqref{eqn::Zintuit}, we define $Z\pcar{i}$ as the solution to the following equations:

\begin{align}
\label{eqn::Zevolve}
Z\pcar{i}(s+1) &= \Lambda_0(Y\pcar{i}(s)) Z\pcar{i}(s) + \sum_{k=1}^m \Lambda_k(Y\pcar{i}(s)) L\pcar{i}_k(s) + \xi_i(s),\\
\label{eqn::Lidef}
L\pcar{i}_k(s)&=\frac{\exmu{\mu_s^2}{Z'(s)B_{s,k}(Z\pcar{i}[s],Z'[s],Y\pcar{i}[s],Y'[s])\middle|Z\pcar{i}[s],Y\pcar{i}[s]}}{\exmu{\mu_s^2}{B_{s,k}(Z\pcar{i}[s],Y\pcar{i}[s],Z'[s],Y'[s])\middle|Z\pcar{i}[s],Y\pcar{i}[s]}}\\
\label{eqn::mu2def}
\mu_s^2 &\defeq \law(Z\pcar{i}[s],Y\pcar{i}[s],Z'[s],Y'[s]) \defeq \mu_s\otimes \mu_s,
\end{align}
where for each $s$, $\xi_{1:r}(s)$ is independent of $\Tmc^r_s$, $\{\xi_i(s)\}_{i\in [1:r],s\in[t]}$ are i.i.d., and $\mu_s = \law(Z\pcar{1}[s],Y\pcar{1}[s])$. To define $B_{s,k}$ above, fix any $(z_{\{1,2\}},y_{\{1,2\}}) \in \mf_s(\R^d)^2\times \mf_s([1:q])^2$. Let $X$ be a Markov process on the set $[m]$ with the following initial condition and time-inhomogenous transition kernel:
\begin{align*}
\PP(X(0) = a) &= B_0(z_{\{1,2\}}(0),y_{\{1,2\}}(0),a),\\
\PP(X(t+1) = a|X(t) = b) &= B(z_{\{1,2\}}(t+1),y_{\{1,2\}}(t+1),b,a).
\end{align*}
Then
\begin{equation}
\label{eqn::Bevolve}
B_{s,k}(z_{\{1,2\}}[s],y_{\{1,2\}}[s]) \defeq \PP\left(X(s) = k\right).
\end{equation}
\ind Lastly, for each $t \in \N_0$, $(Y\pcar{i}(t))_{i \in [1:r]}$ is a collection of mutually conditionally independent $[1:q]$-random variables given $\Gmc^{Y,r}_t$ defined via the following initial condition and transition kernel
\begin{align}
\PP(Y\pcar{i}(t) = y|\Gmc^{Y,r}_{t}) = \begin{cases}
P_0(Z\pcar{i}(0),y) & \te{ if } t = 0,\\
P(Z\pcar{i}(t),Y\pcar{i}(t-1),y) &\te{ if } t > 0.
\end{cases}\label{eqn::Yevolve}
\end{align}

\ind For each $i\neq j \in [1:r]$ and $s \in [t]$, we can also define $A\pcar{ij}(s)$ to be a $[m]$-valued random variable defined by
\begin{equation}
\label{eqn::Aevolve}
\PP\left(A\pcar{ij}(s)=a\middle|\Gmc^{A,r}_s\right) =\begin{cases}
B\left(Z\pcar{i}(s),Y\pcar{i}(s),Z\pcar{j}(s),Y\pcar{j}(s),A\pcar{ij}(s-1),a\right) &\te{ if } s > 0,\\
B_0\left(Z\pcar{i}(s),Y\pcar{i}(s),Z\pcar{j}(s),Y\pcar{j}(s),a\right) &\te{ if } s = 0.
\end{cases}
\end{equation}

\subsubsection{Conditional Structure of the Limiting Model}
\label{ssssec::cstruct}

We now completely define the joint distribution of the limiting mean-field process $(Z\pcar{1:r}[t],Y\pcar{1:r}[t],A\pcar{1:r,1:r}[t])$.

\begin{definition}
\label{defn::AZlim} (Limiting mean-field process) We say that $(Z\pcar{1:r}[t],Y\pcar{1:r}[t],A\pcar{1:r,1:r}[t])$ is a limiting mean-field process (triple) if the following hold. 
For each $i \in [1:r]$, given the initial state $Z\pcar{i}(0)$, $(Z\pcar{i}[t],Y\pcar{i}[t])$ is the unique (in distribution) solution to \eqref{eqn::Zevolve}-\eqref{eqn::Yevolve}. 
Likewise, for each $s \in [t]$ and $i\neq j \in [1:r]$, $A\pcar{ij}(s)$ is a $[m]$-valued random variable satisfying \eqref{eqn::Aevolve}. For any $(i,j) \in [1:r]^2$, $A\pcar{ij}[t] = A\pcar{ji}[t]$ and $A\pcar{ii}(s) = 0$ for all $s \in [t]$. Lastly, $(Z\pcar{i}[t],Y\pcar{i}[t])_{i \in [1:r]}$ are i.i.d. and the joint distribution of $(Z\pcar{1:r}[t],Y\pcar{1:r}[t],A\pcar{1:r,1:r}[t])$ possesses either of the following (equivalent, as proven in Proposition \ref{prop::welldef}) conditional structures:
\begin{enumerate}[label = (\alph*)]
\item For each $s \in [t]$, the $[m]$ random variables in $\{A\pcar{ij}(s)\}_{(i,j) \in \incs_r}$ are mutually conditionally independent given $\Gmc^{A,r}_s$.
\item The $\mf_t([m])$-random elements in $\{A\pcar{ij}[t]\}_{(i,j) \in \incs_r}$ are mutually conditionally independent given $\Gmc^r_t$. In addition, for $(i,j) \in \incs_r$, $A\pcar{ij}[t]$ is a conditional Markov chain given $\Gmc^r_t$ with initial distribution and transition kernel given by
\begin{align}
\PP\left(A\pcar{ij}(s)=a\middle|\Gmc^r_t,A\pcar{ij}[s-1]\right) &= B\left(Z\pcar{i}(s),Y\pcar{i}(s),Z\pcar{j}(s),Y\pcar{j}(s),A\pcar{ij}(s-1),a\right)\te{ for }s > 0,\label{eqn::Atrans}\\
\PP\left(A\pcar{ij}(0)=a\middle|\Gmc^r_t\right) &= B_0\left(Z\pcar{i}(0),Y\pcar{i}(0),Z\pcar{j}(0),Y\pcar{j}(0),a\right).\label{eqn::Ainit}
\end{align}
\item For any $s \in [t]$, the $\mf_s([m])$-random elements in $\{A\pcar{ij}[s]\}_{(i,j) \in \incs_r}$ are mutually conditionally independent given $\Gmc^r_s$.
\end{enumerate}
\end{definition}

Throughout the article, when we reference Definition \ref{defn::AZlim}(a) (or (b) or (c)), we include the first part of the definition in the statement regarding equations \eqref{eqn::Zevolve}-\eqref{eqn::Aevolve}. Proposition \ref{prop::welldef} shows that Definition \ref{defn::AZlim} consistently and completely characterizes the joint distribution of $(Z\pcar{1:r}[t],Y\pcar{1:r}[t],A\pcar{1:r,1:r}[t])$ and that Definition \ref{defn::AZlim}(a),  \ref{defn::AZlim}(b)  and  \ref{defn::AZlim}(c) are indeed equivalent.

\begin{proposition}[Conditional Structure of the Limiting Model]
\label{prop::welldef}
Definitions \ref{defn::AZlim}(a), \ref{defn::AZlim}(b), and \ref{defn::AZlim}(c) individually characterize the distribution of $(Z\pcar{1:r}[t],Y\pcar{1:r}[t],A\pcar{1:r,1:r}[t])$. In addition, the distributions described in all three parts of Definition \ref{defn::AZlim} are equivalent.
\end{proposition}

Definition \ref{defn::AZlim} implies the following additional structure:

\begin{corollary}
\label{coro::morestruct}
For every $s \in [t]$, the random elements $A\pcar{ij}(s)$, $(i,j) \in \incs_r$, are conditionally independent given $\Hmc^{A,r}_s$. In addition, for all $a \in [m]$,
\begin{equation}
\label{eqn::exAH}
\PP\left(A\pcar{ij}(s)=a\middle|\Hmc^{A,r}_s\right) = \PP\left(A\pcar{ij}(s)=a\middle|\Gmc^{A,r}_s\right).
\end{equation}
Lastly, for any $s \in [t]$,
\begin{equation}
\label{eqn::exAG}
\PP\left(A\pcar{ij}(s)=a\middle|\Gmc^r_s\right) = B_{s,a}\left(Z\pcar{i,j}[s],Y\pcar{i,j}[s]\right).
\end{equation}
\end{corollary}

\subsubsection{Convergence Results}
\label{ssssec::mainres}

We now rigorously state the results described in Section \ref{ssssec::intuition} in terms of the limiting mean-field process described in Section \ref{ssssec::limit}. We additionally describe some hydrodynamic limits which follow as a consequence of our main result below. 

\ind We begin by stating our main result: the limiting distribution of the latent opinions of a sample of size $k$ and the resulting subnetwork. We prove this result in Section \ref{ssec::mainpf}.

\begin{theorem}[Asymptotic Distribution of a Random Sample]
\label{thm::Sampleconv}
Suppose Assumptions \ref{assu::cond}, \ref{assu::init}, \ref{assu::ui} and \ref{assu::Bbds} hold. Fix any $t \in \N_0$ and for each $n \in \N$, let $M^n_r = \{m^n_1,\dots,m^n_r\}$ be a uniform random sample (with or without replacement) of size $r$ from the set $[1:n]$ so that $Z^n_{m^n_{1:r}}$ is a sample of size $r$ from the population of latent opinions in our $n$-agent model. Then, the sample latent positions, particle types and the associated network trajectory converge in distribution:
\[\left(Z^n_{m^n_{1:r}}[t],Y^n_{m^n_{1:r}}[t],A^n_{m^n_{1:r},m^n_{1:r}}[t]\right) \Rightarrow \left(Z\pcar{1:r}[t],Y\pcar{1:r}[t],A\pcar{1:r,1:r}[t]\right),\]
where $\left(Z\pcar{1:r}[t],Y\pcar{1:r}[t],A\pcar{1:r,1:r}[t]\right)$ satisfy Definition \ref{defn::AZlim}.
\end{theorem}

We now provide a few hydrodynamic and conditional hydrodynamic limits. Using the notation from Definition \ref{defn::AZlim}, recall that for any $t \in \N$, $\mu_t = \law(Z\pcar{1}[t])$. Additionally, define
\begin{equation}
\label{eqn::muatdef}
\mu\pcar{i}_{A,t} \defeq \law\left(Z\pcar{i}[t],Y\pcar{i}[t],Z'[t],Y'[t],A'[t]\middle|Z\pcar{i}[t],Y\pcar{i}[t]\right),
\end{equation}
where $(Z\pcar{i},Y\pcar{i},Z',Y',A') \deq (Z\pcar{1},Y\pcar{1},Z\pcar{2},Y\pcar{2},A\pcar{12})$. 
\begin{theorem}[Hydrodynamic Limits]
\label{thm::hydro}
Suppose Assumptions \ref{assu::cond}, \ref{assu::init}, \ref{assu::ui} and \ref{assu::Bbds} hold. Then for any $t \in \N_0$, the following hydrodynamic convergences hold:
\begin{equation}
\label{eqn::muntdef}
\mu^n_t \defeq \frac{1}{n}\sum_{i=1}^n \delta_{Z^n_i[t],Y^n_i[t]} \to \mu_t\te{ in probability},
\end{equation}
and for any $i \in \N$,
\begin{equation}
\label{eqn::munaitdef}
(Z^n_i[t],Y^n_i[t],\mu^n_{A,i,t}) \defeq \left(Z^n_i[t],Y^n_i[t],\frac{1}{n}\sum_{j=1}^n \delta_{Z^n_i[t],Y^n_i[t],Z^n_j[t],Y^n_j[t],A^n_{ij}[t]}\right) \Rightarrow (Z\pcar{i}[t],Y\pcar{i}[t],\mu\pcar{i}_{A,t}).
\end{equation}
\end{theorem}
\begin{proof}
This is just a direct application of Theorem \ref{thm::Sampleconv} and a conditional propagation of chaos result \cite[Proposition 5.5]{GanSpiSus25} (restated in the appendix as Proposition \ref{prop::condconvres}).
\end{proof}
\begin{remark}
\label{rmk::munaitdefcpropchaos}
Note that \eqref{eqn::muntdef} of Theorem \ref{thm::hydro} is a standard propagation of chaos result. Likewise, we may regard \eqref{eqn::munaitdef} as a \emph{conditional} propagation of chaos result as it describes a hydrodynamic limit in which the limiting measure is a conditional law. We discuss the concept of conditional propagation of chaos in more detail in Section \ref{sec::condpropchaos}.
\end{remark}

\section{Graphon and Multigraphon Convergence of the Latent Particle Network Process}
\label{ssec::Gconv}

In this section, we examine the interaction networks generated by the model (with adjacency matrix $A^n(t)$) and establish appropriate convergence theorems regarding these graphs. For completeness, we start with a standard review of dense graph convergence in graphon space. We then show that for each $k \in [m]$, the networks $G^{n,k}$ representing interactions of type $k$ converge in a graphon sense and that the network trajectories (collection of all interaction networks at all times) converge in a multiplexon sense described in Section \ref{sssec::trajconv} below. We also show that the graphon limit of the interaction network between agents is determined by the functions $B_{s,a}$ defined in \eqref{eqn::Bsdef}, and the multiplexon limit of the interaction network trajectory is described by the conditional Markov chain outlined in Definition \ref{defn::AZlim}(b). In this section, we review standard definitions, keeping the notation as in \cite{GanSpiSus25} for continuity and consistency purposes, while we emphasize the main differences of the two articles.


\subsection{Left-Convergence of Graphs and Graphon Convergence}
\label{sssec::graphon}

\ind In this work, we work with the notion of left-convergence, which we recall here.  Let $H = (V_H,E_H)$ and $G = (V_G,E_G)$ be two graphs. Then a \emph{homomorphism} from $H$ to $G$ is a map $\phi: V_H\to V_G$ such that if $\{u,v\} \in E_H$, then $\{\phi(u),\phi(v)\} \in E_G$. If $\te{Hom}(H,G)$ is the set of homomorphisms from $H$ to $G$, then the \emph{homomorphism density} of $H$ in $G$ is the function
\[t(H,G) \defeq \frac{\left|\te{Hom}(H,G)\right|}{|V_G|^{|V_H|}}.\]

\ind For each $n \in \N$, let $G_n$ be a graph with $n$ vertices. Then the sequence $\{G_n\}_{n \in \N}$ is said to be \emph{left-convergent} if for all simple, finite graphs $H$, $\lim_{n\to\infty} t(H,G_n)$ exists.

If the sequence $\{G_n\}_{n \in \N}$ is left-convergent, then its limit can be described in terms of a \emph{graphon} \cite{LovSze06}, which is defined below. Let $\wsp_*$ be the space of bounded, symmetric and measurable functions $W: [0,1]^2 \to \R$ equipped with the \emph{cut norm}:
\[\|W\|_{\square} \defeq \sup_{S,T\in\borel([0,1])} \left|\int_{S\times T} W(x,y)\,dx\,dy\right|.\]
Let $\wsp = \{W \in \wsp_*: 0\leq W \leq 1\}$. Then $\wsp$ is a complete, separable metric space, and any $W \in \wsp$ is called a \emph{graphon}. Any finite graph $G = (V,E)$ can be associated with its \emph{empirical graphon} $W^G \in \wsp$ given by
\[W^G(x,y) = \indic{\{\lceil xn\rceil,\lceil yn\rceil\} \in E},\]
where $n = |V|$. Furthermore, the function $t(H,\cdot)$ can be expanded to the space of graphons in the following manner:
\[t(H,W) = \int_{[0,1]^{|V_H|}} \prod_{\{i,j\} \in E_H} W(x_i,x_j)\,dx_1,\dots,\,dx_{|V_H|}.\]
It's easily verified that for any finite graph $G$, $t(H,G) = t(H,W^G)$. 

\ind A map $\psi$ is called a measure-preserving transformation (m.p.t.) if for any Borel-measurable $A \subset [0,1]$, $A$ and $\psi^{-1}(A)$ have the same Lebesgue measure. Two graphons $W_1$ and $W_2$ are said to be \emph{equivalent} ($W_1\sim W_2$) if $W_2$ lies in the closure of the set
\[\{W: W(x,y) = W_1^{\psi}(x,y)\defeq W_1(\psi(x),\psi(y))\te{ for an m.p.t. } \psi:[0,1]\to[0,1]\}.\]
This is analogous to describing two graphs as equivalent if they are isomorphic. 

\ind Define the quotient space $\alt{\wsp}\defeq \wsp/\sim$ and say that $\wsp \ni W \sim \alt{W} \in \alt{\wsp}$ if $W$ is a representative graphon of the equivalence class $\alt{W}$. Then $\alt{\wsp}$ is a compact metric space with respect to the metric
\[\delta_{\square}(W_1,W_2) \defeq \inf_{\substack{\psi: [0,1]\to [0,1]\\\te{measure-preserving}}} \|W_1 - W_2^{\psi}\|_{\square}.\]
It can be shown that for $W_1,W_2\sim \alt{W}\in\alt{\wsp}$ and any graph $H$, $t(H,W_1) = t(H,W_2)$, so the homomorphism density $t(H,\alt{W})$ is a well-defined quantity. With this background, we may now state the following seminal result:

\begin{proposition}
\label{prop::graphonleft}
\cite[Theorem 11.5]{Lov12} A sequence of graphs $\{G_n\}_{n\in \N}$ is left-convergent if and only if there exists a graphon class $\alt{W} \in \alt{\wsp}$ such that
\[\lim_{n\to\infty} \alt{W}^{G_n} = \alt{W}.\]
In addition, for all finite, simple graphs $H$,
\[\lim_{n\to\infty} t(H,G_n) = t(H,\alt{W}).\]
\end{proposition}

With this, we can state our first graph convergence result. Let $\lambda$ be the Lebesgue measure.

We now look at the graph generated by a single type of interaction at a single time. Our measure-preserving transformation also needs to take the particle types into account.
\begin{theorem}[Graphon Limit]
\label{thm::Antconv}
Suppose that Assumptions \ref{assu::cond}-\ref{assu::ui} and \ref{assu::Bbds} hold, and fix any $t \in \N_0$ and $k \in [m]$. For each $n \in \N$, let $G_{n,k}$ be the random graph consisting of all connections of type $k$. That is, let $G_{n,k}$ be the random graph with adjacency matrix $A^{n,k}(t)\defeq (\indic{A^n_{ij}(t)=k})_{i,j=1}^n$. For any $\lambda$-$\mu_t$ measure-preserving transformation $\theta_t: [0,1]\to \mf_t(\R^d)\times\mf_t([1:q])$, let $W_k$ be the graphon defined by
\[W_k(u_1,u_2) = B_{t,k}\left(\theta_t(u_1),\theta_t(u_2)\right)\te{ for all } u_1,u_2 \in [0,1].\]

Then the equivalence class $\alt{W}_k$ of $W_k$ does not depend on our choice of $\lambda$-$\mu_t$ measure-preserving transformations $\theta_t$. In addition, $\alt{W}^{G_{n,k}} \to \alt{W}_k$ in probability in $\alt{\wsp}$.
\end{theorem}

\begin{remark}
\label{rmk::specase}
Theorem \ref{thm::Antconv} is a special case of Theorem \ref{thm::Antmultconv} and Corollary \ref{coro::Wdefthm}. See Remark \ref{rmk::thmantconv} for details. When $m = q = 1$, this reduces to Theorem 3.4 of \cite{GanSpiSus25}.
\end{remark}

\begin{remark}[Existence of a Measure-Preserving Transformation]
\label{rmk::thetatexists}
The argument that $\theta_t$ exists is nearly identical to the argument presented in \cite[Remark 3.6]{GanSpiSus25}.
\end{remark}

\subsection{Convergence of Latent Network Trajectories as Probability Graphons}
\label{sssec::trajconv}

Let $t\in \N_0$ be fixed. For each $n \in \N, s \in [t]$ and $k \in [m]$, define the graph $G_{n,k}(s) = ([1:n],E_{n,k}(s))$ to be the graph with adjacency matrix $A^{n,k}(s)$. Given Theorem \ref{thm::Antconv}, it  follows that
\begin{equation}
\label{eqn::badconv}
\alt{W}^{G_{n,k}[t]} \defeq \left(\alt{W}^{G_{n,k}(s)}\right)_{s \in [t],k\in[m]} \to \alt{W}[s] \defeq (\alt{W}_k(s))_{s \in [t],k\in[m]} \te{ in probability,}
\end{equation}
where $W_k(s)$ is the graphon described in Theorem \ref{thm::Antconv}. However, this result is somewhat unsatisfying. For each $n$, the vertices of $G_{n,k}[t]$ all correspond to a fixed particle. However, \eqref{eqn::badconv} only tells us that there is a way to relabel the vertices of each graph $G_{n,k}(s)$ to get a different graph $G'_{n,k}(s)$ such that
\[W^{G'_{n,k}(s)} \to W_k(s) \te{ for all }s \in [t]\te{ and } k \in [m],\]
where for each $s$ and $k$, $W_k(s)$ is some representative graphon of $\alt{W}_k(s)$. The problem is that the relabeling may differ for each value of $s$ and $k$. This means that the latent particle associated with the vertex labeled ``1'' may be different in $G'_{n,k}(s)$ and $G'_{n,k}(s')$ for $s \neq s'$. As a result, $\alt{W}[s]$ from \eqref{eqn::badconv} does not properly capture heterogeneities in the joint structure of $G_{n,k}$ at different times and for different interaction types.

\ind To handle this, we treat the graph trajectories $G_n[t]$ as \emph{multiplexes}. A multiplex is a collection of graphs sharing a vertex set. Each graph is called a \emph{layer} of the multiplex network. For example, $G_n[t]$ is a multiplex network on the vertex set $[1:n]$ with $(t+1)(m+1)$ layers indexed by $(s,k)$ where for all $s \in [t]$ and $k\in [m]$, the $(s,k)$ layer of the multiplex consists of the edges in $G_{n,k}(s)$. 

\begin{definition}
\label{defn::multdecomp}
For any $(t+1)(m+1)$-layer multiplex $H[t] = (V_H,(E_H(s,k))_{s,k=1}^{t,m})$ and any $S\subseteq [t]\times[m]$, let $H^{S} = (V_H,E_H^{S})$ where
\[E_H^{S} = \cap_{(s,k) \in {S}} E_H(s,k).\]
So $E_H^{S}$ contains all edges that lie in $E_H(s,k)$ for all $(s,k)\in {S}$.
\end{definition}

\ind Next we define multigraphons, which are a special class of probability-graphons recently introduced by Abraham, Delmas and Weibel \cite{AbrDelWei23}. For any finite set $\Smc$, an $\Smc$-layer multigraphon is a vector of graphons indexed by the subsets of $\Smc$: $\mathbf{W} = (W^S)_{S\subseteq \Smc,S\neq\emptyset}$. Let $\Wmc[\Smc]$ be the space of $\Smc$-indexed multigraphons. Then we can define a pseudo-metric $\delta_{\square}[\Smc]$ on $\Wmc[\Smc]$ by
\[\delta_{\square}[\Smc](\mathbf{W}_1,\mathbf{W}_2) = \inf_{\substack{\psi: [0,1]\to[0,1]\\\te{measure-preserving}}}\sum_{\substack{S \subseteq \Smc\\ S\neq \emptyset}} \|W_1^S - (W_2^S)^{\psi}\|_{\square}.\]
Letting $\mathbf{W}_1\sim \mathbf{W}_2$ if and only if $\delta_{\square}[\Smc](\mathbf{W}_1,\mathbf{W}_2) = 0$, let $\alt{\Wmc}[\Smc] = \Wmc[\Smc]/\sim$ and note that $\delta_{\square}[\Smc]$ is a metric on $\alt{\Wmc}[\Smc] = \Wmc[\Smc]/\sim$. Under this definition, the vertices of each layer of the multigraphon are labeled consistently.\footnote{\label{ft::prob}Using the terminology and notation of \cite{AbrDelWei23}, a $\Smc$-indexed multigraphon $\bm{W}$ is topologically equivalent to a probability-graphon $W$ on the power set $\bm{Z} \defeq \ps(\Smc)$ equipped with the discrete topology. More specifically, for any $(x,y) \in [0,1]^2$ and $S \subseteq \Smc, S\neq \emptyset$, $W$ is defined by
\begin{equation}
\label{eq::mpgraphonequiv}
W(x,y;\cdot) = \mu_{x,y} \in \P(\ps(\Smc))\te{ where } \mu_{x,y}(\{S'\subseteq \Smc: S\subseteq S'\}) = W^S(x,y).
\end{equation}
By using this representation, it is easily shown that the metric $\delta_{\square}[\Smc]$ is equivalent to the metric $\delta_{\square,\Pmc}$ introduced in \cite{AbrDelWei23}:
\[\delta_{\square,\Pmc}(\alt{W}_1,\alt{W}_2)\leq \delta_{\square}[\Smc](\alt{\bm{W}}_1,\alt{\bm{W}}_2) \leq 2^{(t+1)(m+1)}\delta_{\square,\Pmc}(\alt{W}_1,\alt{W}_2)\]
for all multigraphon classes (or probability-graphon classes) $\alt{\bm{W}}_1$ and $\alt{\bm{W}}_2$ in $\alt{\Wmc}[\Smc]$, where for $i = 1,2$, $W_i$ is defined via \eqref{eq::mpgraphonequiv} with respect to $\bm{W}_i$ and $\alt{W}_i$ is the probability-graphon class represented by $W_i$.
}

\ind The empirical multigraphon of a multiplex $G[t]$ is given by
\[\mathbf{W}^{G[t]} = \left(W^{G^S}\right)_{S\subseteq [t]\times[m],S\neq \emptyset}.\]

We now state the following theorem which extends Theorem \ref{thm::Antconv}. The proof of this theorem follows along the lines of Theorem 3.9 of \cite{GanSpiSus25} and we do not repeat it here.
\begin{theorem}[Multigraphon Convergence]
\label{thm::Antmultconv} Suppose Assumptions \ref{assu::cond}-\ref{assu::ui} and \ref{assu::Bbds} hold at time zero and fix any $t \in \N_0$. For each $n \in \N$, let $G_n[t]$ be the random multiplex such that for each $s \in [t]$ and $k \in [m]$, $G_{n,k}(s)$ is the random graph with adjacency matrix $A^{n,k}(s)$. Then there exists a multigraphon $\mathbf{W}$ (given explicitly in \eqref{eqn::Wmultdef} below) such that
\[\alt{\mathbf{W}}^{G_n[t]} \to \alt{\mathbf{W}} \te{ in probability.}\]
\end{theorem}

We can also characterize the multigraphon $\alt{\mathbf{W}}$. Consider the measure-preserving transformation $\theta_t$ defined in Theorem \ref{thm::Antconv}. Recall the time-inhomogeneous Markov chain $X$ defined above \eqref{eqn::Bevolve}. For any $(z_{1,2},y_{1,2}) \in \mf_t(\R^d)^2\times \mf_t([1:q])$, let $B_{z_{1,2},y_{1,2}} = X$, where $X$. 

\begin{corollary}[Multigraphon Limit]
\label{coro::Wdefthm}
Under the conditions of Theorem \ref{thm::Antmultconv}, $\mathbf{W}$ satisfies $\mathbf{W} = (W^S)_{S \subseteq [t], S\neq \emptyset}$, where for each non-empty $S \subseteq [t]$,
\begin{equation}
\label{eqn::Wmultdef}
W^S(u_1,u_2) = \PP\left(B_{\theta_t(u_1),\theta_t(u_2)}(s) = 1 \te{ for all }s \in S\right).
\end{equation}
\end{corollary}
\begin{proof}
The proof is virtually identical to the proof of \cite[Theorem 3.9, Corollary 3.11]{GanSpiSus25}, only now $S = [t]\times [m]$ instead of $[t]$.
\end{proof}

\begin{remark}
\label{rmk::thmantconv}
As a consequence of Corollary \ref{coro::morestruct} and Definition \ref{defn::AZlim}(b), if we set $S = \{(t,k)\}$, then 
\begin{align*}
W^{\{(t,k)\}}(u_1,u_2) &= \PP\left(B_{\theta_t(u_1),\theta_t(u_2)}(t) = k\right) = \PP\left(A\pcar{12}(t) = k\middle|\Gmc^2_t\right)\bigg|_{(Z\pcar{1:2}[t],Y\pcar{1:2}[t])=\theta_t(u_{1:2})}\\
&= B_{t,k}\left(\theta_t(u_1),\theta_t(u_2)\right) = W_k(u_1,u_2).
\end{align*}
This proves that Theorem \ref{thm::Antconv} is a special case of Theorem \ref{thm::Antmultconv} together with Corollary \ref{coro::Wdefthm}.
\end{remark}


\section{Proofs of main results of Section \ref{ssec::assures}}
\label{sec::proofs}

In this section, we prove the major theoretical statements of the paper. In Section \ref{ssec::limstruct}, we show that the limiting mean-field process of Definition \ref{defn::AZlim} is well-defined, and we rigorously establish the rich conditional structure of the limiting process established in Definition \ref{defn::AZlim} and Corollary \ref{coro::morestruct}. This is utilized in Section \ref{ssec::mainpf} where we prove Theorem \ref{thm::Sampleconv}.

\ind In the introduction, we mentioned the existing article \cite{GanSpiSus25} which describes our model in the special case that $q = m = 1$. For this reason, there are some portions of the upcoming proofs in which arguments from \cite{GanSpiSus25} apply with minimal or no changes. In those  cases, these arguments are summarized and omitted for conciseness.  However, the case studied in this paper has many additional difficulties which we overcome here. The proofs focus mainly on those arguments.

\subsection{Properties of the Limiting Process}
\label{ssec::limstruct}

Before proving the main results of this section, we establish the following useful lemma. This lemma is analogous to \cite[Lemma 6.2]{GanSpiSus25}. However, the proof of the lemma is significantly different due to the existence of particle types $Y\pcar{1:r}$ which are assumed to be trivial in \cite{GanSpiSus25}.  

\begin{lemma}
\label{lem::futind}
If $(Z\pcar{1:r}[t],Y\pcar{1:r}[t],A\pcar{1:r,1:r}[t])$ satisfies the conditions laid out in Definition \ref{defn::AZlim}(a),(b) or (c), then for any $s \in [t]$, $A\pcar{1:r,1:r}[s]$ is conditionally independent of $\Gmc^r_t$ given $\Gmc^r_s$.
\end{lemma}
\begin{proof}
For this proof, recall the following $\sigma$-algebras defined at the beginning of Section \ref{ssssec::limit}. For each $r \in \N$ and $s \in \N_0$, $\Tmc^r_s = \sigma(Z^{1:r}[s],Y^{1:r}[s],A^{1:r,1:r}[s])$ and $\Gmc^r_s = \sigma(Z^{1:r}[s],Y^{1:r}[s])$.

\ind Fix any $s' \geq s$. By \eqref{eqn::Lidef}, $L\pcar{i}(s')$ is $\Gmc^r_{s'}$-measurable for any $i \in [1:r]$. It follows by \eqref{eqn::Zevolve} that $Z\pcar{i}(s'+1) - \xi_i(s')$ is $\Gmc^r_{s'}$-measurable for each $i \in [1:r]$, so $Z\pcar{i}[s'+1]$ is $\Gmc^r_{s'}\vee \sigma(\xi_{1:r}(s'))$-measurable. Let $g: \mf_{s'+1}\left((\R^d)^r\right)\times \mf_{s'+1}\left([1:q]^r\right)\to \R$ be any bounded, measurable function. 
\begin{align*}
&\ex{g(Z\pcar{1:r}[s'+1],Y\pcar{1:r}[s'+1])\middle|\Tmc^r_s} \\
&\ind\ind= \ex{\ex{g(Z\pcar{1:r}[s'+1],Y\pcar{1:r}[s'+1])\middle|\Tmc^r_{s'}}\middle|\Tmc^r_s}\\
&\ind\ind=\ex{\ex{\ex{g(Z\pcar{1:r}[s'+1],Y\pcar{1:r}[s'+1])\middle|\Gmc^{Y,r}_{s'+1}}\middle|\Tmc^r_{s'}}\middle|\Tmc^r_s}\\
&\ind\ind=\ex{\ex{\sum_{\bm{y} \in [1:q]^r}g(Z\pcar{1:r}[s'+1],Y\pcar{1:r}[s'],\bm{y})\PP\left(Y\pcar{1:r}(s'+1) = \bm{y}\middle|\Gmc^{Y,r}_{s'+1}\right)\middle|\Tmc^r_{s'}}\middle|\Tmc^r_s}\\
&\ind\ind=\ex{\ex{\sum_{\bm{y} \in [1:q]^r}g(Z\pcar{1:r}[s'+1],Y\pcar{1:r}[s'],\bm{y})\prod_{i=1}^r P(Z\pcar{i}(s'+1),Y\pcar{i}(s'),y_i)\middle|\Tmc^r_{s'}}\middle|\Tmc^r_s},\\
\end{align*}
where the final equality follows from \eqref{eqn::Yevolve}. Noting that the term in the expectation above is $\Gmc^r_{s'}\vee\sigma(\xi_{1:r}(s'))$-measurable, that $\xi_{1:r}(s')$ is independent of $\Tmc^r_{s'}$, and that $\Gmc^r_{s'}\subset \Tmc^r_{s'}$, it follows that 
\[\ex{g(Z\pcar{1:r}[s'+1],Y\pcar{1:r}[s'+1])\middle|\Tmc^r_{s'}} = \ex{g(Z\pcar{1:r}[s'+1],Y\pcar{1:r}[s'+1])\middle|\Gmc^r_{s'}},\]
so, because $\Gmc^r_{s'} = \sigma(Z\pcar{1:r}[s'],Y\pcar{1:r}[s'])$, there must exist a bounded, measurable $\bar{g}: \mf_{s'}\left((\R^d)^r\right)\times \mf_{s'}\left([1:q]^r\right)\to \R$ such that 
\begin{align*}
\ex{g(Z\pcar{1:r}[s'+1],Y\pcar{1:r}[s'+1])\middle|\Tmc^r_s} &= \ex{\ex{g(Z\pcar{1:r}[s'+1],Y\pcar{1:r}[s'+1])\middle|\Gmc^r_{s'}}\middle|\Tmc^r_s}\\
&= \ex{\tbb{\bar{g}}(Z\pcar{1:r}[s'],Y\pcar{1:r}[s'])\middle|\Tmc^r_s}.
\end{align*}

Iterating this argument, there must exist a bounded, measurable function $h: \mf_{s}\left((\R^d)^r\right)\times \mf_{s}\left([1:q]^r\right)\to \R$ such that
\begin{align*}
\ex{g(Z\pcar{1:r}[s'+1],Y\pcar{1:r}[s'+1])\middle|\Tmc^r_s} &= \ex{\ex{h(Z\pcar{1:r}[s'+1],Y\pcar{1:r}[s'+1])\middle|\Gmc^r_s}\middle|\Tmc^r_s}\\
&= \ex{h(Z\pcar{1:r}[s'+1],Y\pcar{1:r}[s'+1])\middle|\Gmc^r_s},
\end{align*}
where we use the fact that $\Gmc^r_s \subseteq \Tmc^r_s.$

Set $s' = t-1$. Then for any bounded, measurable $g: \mf_{t}\left((\R^d)^r\right)\times \mf_{t}\left([1:q]^r\right)\to \R$, $\ex{g(Z\pcar{1:r}[t],Y\pcar{1:r}[t])\middle|\Tmc^r_s}$ is $\Gmc^r_s$ measurable. Since $\Tmc^r_s = \Gmc^r_s \vee \sigma(A\pcar{1:r,1:r}[s])$ and $\Gmc^r_t = \sigma(Z\pcar{1:r}[t],Y\pcar{1:r}[t])$, it follows that for any bounded, $\Gmc^r_t$-measurable random variable $W$,
\[\ex{W\middle|\Gmc^r_s, A\pcar{1:r,1:r}[t]} = \ex{W|\Gmc^r_s},\]
which completes the proof.
\end{proof}

\subsubsection{Proof of Proposition \ref{prop::welldef}}
We first start by noting that Lemmas 6.4-6.6 of \cite{GanSpiSus25} hold in the context of this model as well (with Definition 2.5 of \cite{GanSpiSus25} replaced by Definition \ref{defn::AZlim} of this paper). The proofs of these lemmas in the context of this paper require minimal modification (e.g. the link functions $B$ and $B_0$ are defined differently between the two papers, but their application within the proofs is near identical). Therefore, we omit the proofs of these lemmas within the context of our model.

\ind By Lemmas 6.4-6.6 of \cite{GanSpiSus25}, Definitions \ref{defn::AZlim}(a),(b), and (c) are equivalent. Therefore it suffices to establish that any one of these definitions completely characterizes the limiting process $(Z\pcar{1:r}[t],Y\pcar{1:r}[t],A\pcar{1:r,1:r}[t])$. This is the content of Lemma \ref{lem::achar} and it concludes the proof of Proposition \ref{prop::welldef}.

\begin{lemma}
\label{lem::achar}
The distribution of $(Z\pcar{1:r}[t],Y\pcar{1:r}[t],A\pcar{1:r,1:r}[t])$ is completely characterized by Definition \ref{defn::AZlim}(b).
\end{lemma}
\begin{proof}
The first part of Definition \ref{defn::AZlim} completely characterizes the marginal distribution of $(Z\pcar{1:r}[t],Y\pcar{1:r}[t])$. Moreover, $A\pcar{ij}[t] = A\pcar{ji}[t]$ for all $i,j \in [1:r]$, and $A\pcar{ii}[t] = 0$ for all $i \in [1:r]$. Therefore, it suffices to prove that the conditional distribution of $(A\pcar{ij}[t])_{(i,j) \in \incs_r}$ given $\Gmc^r_t = \sigma(Z\pcar{1:r}[t],Y\pcar{1:r}[t])$ is well-defined, where we recall that $\incs_r = \{(i,j)\in[r]^2: 1\leq i <j\leq r\}$.

\ind Fix any $\bm{a}\defeq (a_{ij}[t])_{(i,j)\in\incs_r} \in (\mf_t([m]))^{\incs_r}$. Then by Definition \ref{defn::AZlim}(b) and \eqref{eqn::Aevolve}:
\begin{align}
\PP&\left((A\pcar{ij}[t])_{(i,j)\in \incs_r} = (a_{ij}[t])_{(i,j)\in\incs_r}\middle|\Gmc^r_t\right)= \prod_{(i,j)\in \incs_r}\PP\left(A\pcar{ij}[t] = a_{ij}[t]\middle|\Gmc^r_t\right)\nonumber\\
&=\prod_{(i,j)\in \incs_r}\PP\left(A\pcar{ij}(0) = a_{ij}(0)\middle|\Gmc^r_t\right)\prod_{s=1}^t
\PP\left(A\pcar{ij}(s) = a_{ij}(s)\middle|\Gmc^r_t,A\pcar{ij}[s-1]=a_{ij}[s-1]\right)\nonumber\\
&=\prod_{(i,j)\in\incs_r}B_0(Z\pcar{i,j}(0),Y\pcar{i,j}(0),a\pcar{ij}(0))\prod_{s=1}^t B(Z\pcar{i,j}(s),Y\pcar{i,j}(s),a\pcar{ij}[s-1:s]).\label{eqn::Atprob}
\end{align}
This concludes the proof of the lemma. 
\end{proof}

\subsubsection{Proof of Corollary \ref{coro::morestruct}}

We start by establishing \eqref{eqn::exAH}. Fix $(i,j) \in \incs_r$, $a \in [m]$, and $s \in [t]$. For $s = 0$, $\Hmc^{A,r}_0 = \Gmc^{A,r}_0$, so \eqref{eqn::exAH} holds. Then for $s > 0$,
\begin{align*}
\PP\left(A\pcar{ij}(s)=a\middle|\Hmc^{A,r}_s\right) &= \ex{\PP\left(A\pcar{ij}(s)=a\middle|\Gmc^{A,r}_s\right)\middle|\Hmc^{A,r}_s}\\
&= \ex{B\left(Z\pcar{i,j}(s),Y\pcar{i,j}(s),A\pcar{ij}(s-1),a\right)\middle|\Hmc^{A,r}_s}\\
&= B\left(Z\pcar{i,j}(s),Y\pcar{i,j}(s),A\pcar{ij}(s-1),a\right)\\
&= \PP\left(A\pcar{ij}(s)=a\middle|\Gmc^{A,r}_s\right).
\end{align*}
By the above computation, Definition \ref{defn::AZlim}(b), and \eqref{eqn::Bevolve},
\begin{align}
\PP\left(A\pcar{ij}(s) = a\middle|\Gmc^{r}_s\right) &= \sum_{\substack{\bm{a}\in \mf_s([m])\\ \bm{a}(s) = a}} B_0(Z\pcar{i,j}(0),Y\pcar{i,j}(0),\bm{a}(0))\prod_{s'=1}^s B(Z\pcar{i,j}(s'),Y\pcar{i,j}(s'),\bm{a}[s'-1:s'])\nonumber\\
&= B_{s,a}(Z\pcar{i,j}[s],Y\pcar{i,j}[s])\te{ a.s..}\label{eqn::Bsadef}
\end{align}
This completes the proof.

\subsubsection{Absolute Continuity of the Limiting Process}
\begin{lemma}
\label{lem::abscont}
For any $t \in \N_0$, $Z\pcar{1:r}[t]$ is absolutely continuous with respect to the Lebesgue measure on $(\mf_t(\R^d))^r$.
\end{lemma}
\begin{proof}
By Definition \ref{defn::AZlim}, $Z\pcar{1}[t],\dots,Z\pcar{r}[t]$ are i.i.d.. It therefore suffices to prove that $Z\pcar{1}[t]$ is absolutely continuous. To do this, we apply an inductive proof. If $t = 0$, then the result follows directly from Assumption \ref{assu::init}(a). 

\ind Suppose that for some $t\geq 0$ we know that $Z\pcar{1}[t]$ is absolutely continuous. We first show that the conditional distribution of $Z\pcar{1}(t+1)$ given $Z\pcar{1}[t]$ is absolutely continuous. By \eqref{eqn::Zevolve}, $Z\pcar{1}(t+1) - \xi_1(t)$ is $\Gmc^1_t$-measurable, so conditioned on $\Gmc^1_t$, it is constant. As stated after \eqref{eqn::mu2def}, $\xi_1(t)$ is independent of $\Gmc^1_t$, and by Assumption \eqref{assu::cond}, $\xi_1(t)$ is absolutely continuous. So, conditioned on $\Gmc^1_t$, $Z\pcar{1}(t+1)$ is the sum of a constant vector and an absolutely continuous random vector and is therefore absolutely continuous as well. This means there exists a measurable function $f: \R^d\times \mf_t(\R^d)\times \mf_t([1:q]) \to \R_+$ such that for any $B\in \borel(\R^d)$,
\[\PP\left(Z\pcar{1}(t+1)\in B\middle|\Gmc^1_t\right) = \int_B f(z;Z\pcar{1}[t],Y\pcar{1}[t])\,dz\te{ a.s.}.\]
Note that $Y\pcar{1}[t]$ is a discrete random element with a finite supporting set. We may therefore define a function $g: \mf_t([1:q])\times \mf_t(\R^d)\to [0,1]$ such that $g(y,Z\pcar{1}[t]) = \PP(Y\pcar{1}[t] = y|Z\pcar{1}[t])$ a.s.. Then,
\begin{align*}
\PP\left(Z\pcar{1}(t+1) \in B\middle|Z\pcar{1}[t]\right) &= \sum_{y \in \mf_t([1:q])} \PP\left(Z\pcar{1}(t+1)\in B, Y\pcar{1}[t] = y\middle|Z\pcar{1}[t]\right)\\
&= \sum_{y \in \mf_t([1:q])}\PP\left(Z\pcar{1}(t+1)\in B\middle|Z\pcar{1}[t],Y\pcar{1}[t] = y\right)\PP\left(Y\pcar{1}[t] = y\middle|Z\pcar{1}[t]\right)\\
&= \int_B \sum_{y \in \mf_t([1:q])}f(z;Z\pcar{1}[t],y)g(y;Z\pcar{1}[t])\,dz.
\end{align*}
This shows that the conditional distribution of $Z\pcar{1}(t+1)$ given $Z\pcar{1}[t]$ is absolutely continuous. This implies that $Z\pcar{1}[t+1]$ is absolutely continuous. The lemma follows by induction.
\end{proof}

\subsection{Proof of Theorem \ref{thm::Sampleconv}}
\label{ssec::mainpf}

In this section, we prove the main sampling theorem, Theorem \ref{thm::Sampleconv}. The proof advances in two parts. First, in Section \ref{sssec::prelje} we show that $(Z^n,Y^n,A^n)$ are \emph{jointly exchangeable} for all $n$. This essentially tells us that the joint distribution of $(Z^n,Y^n,A^n)$ is invariant with respect to permutations of the labelings of the particles. This condition is then leveraged in Section \ref{ssec::mainpffin} to prove the result. Our proof utilizes the conditional propagation of chaos theorem from \cite{GanSpiSus25} (and restated in Section \ref{sec::condpropchaos}) to construct an inductive proof of the result.

\subsubsection{Preliminary Result: Joint Exchangeability} 
\label{sssec::prelje}

Exchangeability is one of the most useful tools for establishing mean-field results. Here, we begin by showing that our process satisfies the following notion of \emph{joint exchangeability} that was also utilized in \cite{GanSpiSus25}:

\begin{definition}[Joint Exchangeability]
\label{def::jointexchangeability}
For any $n \in\N$, let $X = (X_1,\dots,X_n)$ be an $\Xmc^n$-random element, and let $M = (M_{ij})_{i,j=1}^n$ be an $n\times n$ random matrix with entries in $\Ymc$. Then we say that the pair $(X,M)\defeq (X_{1:n},M_{1:n,1:n})$ is \emph{jointly exchangeable} if and only if for any permutation $\sigma \in S_n$,
\[(X_{1:n},M_{1:n,1:n}) \deq (X_{\sigma(1:n)},M_{\sigma(1:n)\sigma(1:n)}).\]
\end{definition}

We now show that our interacting particle system is jointly exchangeable.

\begin{lemma}
\label{lem::jexch}
For any $n \in \N$ and $t \in \N_0$, the process $(Z^n[t],Y^n[t],A^n[t])$ is jointly exchangeable.
\end{lemma}
\begin{proof}
We use an inductive proof. \\

\textbf{Base Case, $t=0$:} Recall from Section \ref{ssec::notation} that $S_n$ is the permutation group on the set $[1:n]$ and $\Fmc^{Y,n}_0 \defeq \sigma(Z^n(0))$. Let $\sigma \in S_n$ be any permutation. For $i \in [1:n]$, fix $A_i \in \borel(\R^d)$ and $y_i \in [1:q]$. Then,
\begin{align*}
\PP\left((Z^n_i(0),Y^n_i(0))_{i=1}^n \in \times_{i=1}^n (A_i\times \{y_i\})\right) &= \ex{\PP\left(Y^n(0) = y\middle|\Fmc^{Y,n}_0\right)\prod_{i=1}^n \indic{Z^n_i(0) \in A_i}}\\
&= \ex{\prod_{i=1}^n \indic{Z^n_i(0) \in A_i}\PP\left(Y^n_i(0) = y_i\middle|\Fmc^{Y,n}_0\right)}\\
&=\ex{\prod_{i=1}^n \indic{Z^n_i(0) \in A_i}P_0(Z^n_i(0),y_i)}\\
&=\ex{\prod_{i=1}^n \indic{Z^n_{\sigma(i)}(0) \in A_i}P_0(Z^n_{\sigma(i)}(0),y_i)}\\
&=\PP\left((Z^n_{\sigma(i)}(0),Y^n_{\sigma(i)}(0))_{i=1}^n \in \times_{i=1}^n (A_i\times \{y_i\})\right)
\end{align*}

This proves that $(Z^n_i(0),Y^n_i(0))_{i=1}^n$ is an exchangeable collection of random elements. A similar proof shows that $(Z^n,Y^n,A^n)$ is jointly exchangeable. That is, for $A_i \in \borel(\R^d\times [1:q])$, $i = 1,\dots,n$, and $a_{ij} \in [m], (i,j) \in \incs_n$, let $A = \times_{i=1}^n A_i$ and $a = \{a_{ij}\}_{i,j=1}^n$ with $a_{ij} = a_{ji}$ and $a_{ii} = 0$ for all $i,j$. Recall also that $\Fmc^{A,n}_0 = \sigma(Z^n(0),Y^n(0))$. Then,
\begin{align*}
\PP&\left((Z^n_{1:n}(0),Y^n_{1:n}(0),A^n_{1:n,1:n}(0))\in A\times \{a\}\right) = \ex{\PP\left(A^n(0) = a\middle|\Fmc^{A,n}_0\right)\prod_{i=1}^n \indic{(Z^n_i(0),Y^n_i(0)) \in A_i}}\\
&\ind =\ex{\prod_{(i,j) \in \incs_n} \PP\left(A^n_{ij}(0) = a_{ij}\middle|\Fmc^{A,n}_0\right)\prod_{k=1}^n \indic{(Z^n_k(0),Y^n_k(0)) \in A_k}}\\
&\ind = \ex{\prod_{(i,j) \in \incs_n} B_0\left(Z^n_{i,j}(0),Y^n_{i,j}(0),a_{ij}\right)\prod_{k=1}^n \indic{(Z^n_k(0),Y^n_k(0)) \in A_k}}\\
&\ind = \ex{\prod_{(i,j) \in \incs_n} B_0\left(Z^n_{\sigma(i),\sigma(j)}(0),Y^n_{\sigma(i),\sigma(j)}(0),a_{ij}\right)\prod_{k=1}^n \indic{(Z^n_{\sigma(k)}(0),Y^n_{\sigma(k)}(0)) \in A_k}}\\
&\ind =\PP\left((Z^n_{\sigma(1):\sigma(n)}(0),Y^n_{\sigma(1):\sigma(n)}(0),A^n_{\sigma(1):\sigma(n),\sigma(1):\sigma(n)}(0))\in A\times \{a\}\right).
\end{align*}

Thus, $(Z^n(0),Y^n(0),A^n(0))$ is jointly exchangeable.\\

\textbf{Inductive proof:} Now suppose that $(Z^n[t],Y^n[t],A^n[t])$ is jointly exchangeable. Then to prove Lemma \ref{lem::jexch}, it suffices to show that $(Z^n[t+1],Y^n[t+1],A^n[t+1])$ is jointly exchangeable. The proof that $(Z^n[t+1],Y^n[t],A^n[t])$ is jointly exchangeable is nearly identical to the proof of \cite[Lemmas A.9,7.7]{GanSpiSus25}, so it is omitted. We now show that $(Z^n[t+1],Y^n[t+1],A^n[t])$ is jointly exchangeable.  Fix any $n \in\N$. For each $(\bm{y},\bm{a}) \in [1:q]^n\times[m]^{n\times n}$, define
\begin{align*}
&\PP\left(Y^n(t+1) =\bm{y}, A^n_{ij}(t+1)=\bm{a}\middle|\Fmc^{Y,n}_{t+1}\right)\\
&\ind\ind = \PP\left(Y^n(t+1) = \bf{y}\middle|\Fmc^{Y,n}_{t+1}\right)\PP\left(A^n(t+1)=\bf{a}\middle|\Fmc^{A,n}_{t+1},Y^n(t+1)=\bf{y}\right)\\
&\ind\ind= \prod_{i=1}^n\PP\left(Y^n_{i}(t+1) = y_i\middle|\Fmc^{Y,n}_{t+1}\right)\prod_{(i,j) \in \incs_n}\PP\left(A^n_{ij}(t+1)=a_{ij}\middle|\Fmc^{A,n}_{t+1},Y^n(t+1)=\bf{y}\right)\\
&\ind\ind= \prod_{i=1}^n P(Z^n_i(t+1),Y^n_i(t),y_i)\prod_{(i,j) \in \incs{n}} B(Z^n_{ij}(t+1),y_{ij},A^n_{ij}(t),a_{ij})\\
&\ind\ind= \prod_{i=1}^n P(Z^n_i(t+1),Y^n_i(t),y_i)\prod_{(i,j) \in \incs{n}} B(Z^n_{ij}(t+1),y_{ij},A^n_{ij}(t),a_{ij}).
\end{align*}

Define the space $[m]^{n\times n}_{\te{sym0}}$ to be the space of $n$ by $n$ symmetric matrices with entries in $[m]$ and whose diagonal is the zero vector. It directly follows from the joint exchageability of $(Z^n[t+1],Y^n[t],A^n[t])$ that for any $g: \mf_{t+1}(\R^d)^n\times \mf_{t+1}([1:q])^n\times \mf_{t+1}([m])^{n\times n}$,\begin{align*}
&\ex{g(Z^n_{1:n}[t+1],Y^n_{1:n}[t+1],A^n_{1:n,1:n}[t+1])}\\
&= \ex{\ex{g(Z^n_{1:n}[t+1],Y^n_{1:n}[t+1],A^n_{1:n,1:n}[t+1])\middle|\Fmc^{Y,n}_{t+1}}}\\
&= \sum_{(\bf{y},\bf{a})\in [1:q]^n\times [m]^{n\times n}_{\te{sym0}}} \ex{g(Z^n_{1:n}[t+1],Y^n_{1:n}[t],y_{1:n},A^n_{1:n,1:n}[t],a_{1:n,1:n})\PP\left(Y^n(t+1) = \bf{y}, A^n(t+1) = \bf{a}\middle|\Fmc^{n,Y}_{t+1}\right)}\\
&= \sum_{(\bf{y},\bf{a})\in [1:q]^n\times [m]^{n\times n}_{\te{sym0}}} \mathbb{E}\Bigg[g(Z^n_{1:n}[t+1],Y^n_{1:n}[t],y_{1:n},A^n_{1:n,1:n}[t],a_{1:n,1:n})\prod_{i=1}^n P(Z^n_i(t+1),Y^n_i(t),y_i)\\
&\ind \ind \prod_{(i,j) \in \incs{n}} B(Z^n_{ij}(t+1),y_{ij},A^n_{ij}(t),a_{ij})\Bigg]\\
&= \sum_{(\bf{y},\bf{a})\in [1:q]^n\times [m]^{n\times n}_{\te{sym0}}} \mathbb{E}\Bigg[g(Z^n_{\sigma(1:n)}[t+1],Y^n_{\sigma(1:n)}[t],y_{1:n},A^n_{\sigma(1:n),\sigma(1:n)}[t],a_{1:n,1:n})\prod_{i=1}^n P(Z^n_{\sigma(i)}(t+1),Y^n_{\sigma(i)}(t),y_i)\\
&\ind \ind \prod_{(i,j) \in \incs{n}} B(Z^n_{\sigma(i)\sigma(j)}(t+1),y_{ij},A^n_{\sigma(ij)}(t),a_{ij})\Bigg]\\
&=\ex{g(Z^n_{\sigma(1:n)}[t+1],Y^n_{\sigma(1:n)}[t+1],A^n_{\sigma(1:n),\sigma(1:n)}[t+1])}.
\end{align*}
This confirms the joint exchangeability of $\left(Z^n[t+1],Y^n[t+1],A^n[t+1]\right)$.
\end{proof}

\subsubsection{Proof of Theorem \ref{thm::Sampleconv}}
\label{ssec::mainpffin}

Again, we take an inductive approach. \\

\textbf{Base case, $t = 0$:} Fix any $r \in\N$, and assume $n > r$. Fix any $\bm{a} \defeq (a_{ij})_{(i,j)\in \incs_r} \in [m]^{\incs_r}$ and $\bm{y} \defeq (y_i)_{i=1}^r \in [1:q]^r$. Then \eqref{eqn::Anevolve} and Assumption \ref{assu::cond}(b) imply 
\begin{align*}
&\PP\left(Y^n_{1:r}(0) = \bm{y},(A^n_{ij}(0))_{(i,j)\in \incs_r} = \bm{a}\middle|Z^n_{1:r}(0)\right) \\
&\ind = \ex{\PP\left(Y^n_{1:r}(0) = \bm{y},(A^n_{ij}(0))_{(i,j)\in \incs_r} = \bm{a}\middle|\Fmc^{Y,n}_0\right)\middle|Z^n_{1:r}(0)}\\ 
&= \ex{\prod_{(i,j)\in \incs_r}B_0\left(Z^n_{\{i,j\}}(0),y_{\{i,j\}},a_{ij}\right)\prod_{k=1}^r P_0(Z^n_k(0),y_k)\middle| Z^n_{[1:r]}(0)}\\
&= \prod_{(i,j)\in \incs_r}B_0\left(Z^n_{\{i,j\}}(0),y_{\{i,j\}},a_{ij}\right)\prod_{k=1}^r P_0(Z^n_k(0),y_k).
\end{align*}
Thus, the probability above is an a.e. continuous function of $Z^n_{1:r}(0)$. Now, by \cite[Proposition 2.2(i)]{Szn91}, 
\[Z^n_{1:r}(0) \Rightarrow Z\pcar{1:r}(0),\]
where $(Z\pcar{i}(0))_{i=1}^r$ are i.i.d.. Consider any bounded, continuous $f: (\R^d)^r\times [1:q]^r\times [m]^{\incs_r} \to \R$. Then,

\begin{align*}
\lim_{n\to\infty}& \ex{f\left(Z^n_{1:r}(0),Y^n_{1:r}(0),A^n_{1:r,1:r}(0)\right)} \\
&= \lim_{n\to\infty}\ex{\sum_{\substack{\bm{y} \in [1:q]^r\\ \bm{a} \in [m]^{r\times r}\te{ sym. }, \te{diag}(\bm{a}) = 0}} f(Z^n_{1:r}(0),\bm{y},\bm{a})\PP\left(Y^n_{1:r}(0) = \bm{y},A^n_{1:r,1:r}(0) = \bm{a}\middle|Z^n_{1:r}(0)\right)}\\
&= \lim_{n\to\infty}\ex{\sum_{\substack{\bm{y} \in [1:q]^r\\ \bm{a} \in [m]^{r\times r}\te{ sym. }, \te{diag}(\bm{a}) = 0}} f(Z^n_{1:r}(0),\bm{y},\bm{a})\prod_{(i,j)\in \incs_r}B_0\left(Z^n_{\{i,j\}}(0),y_{\{i,j\}},a_{ij}\right)\prod_{k=1}^r P_0(Z^n_k(0),y_k)}\\
&=\ex{\sum_{\substack{\bm{y} \in [1:q]^r\\ \bm{a} \in [m]^{r\times r}\te{ sym. }, \te{diag}(\bm{a}) = 0}} f(Z\pcar{1:r}(0),\bm{y},\bm{a})\prod_{(i,j)\in \incs_r}B_0\left(Z\pcar{i,j}(0),y_{\{i,j\}},a_{ij}\right)\prod_{k=1}^r P_0(Z\pcar{k}(0),y_k)}\\
&= \ex{f\left(Z\pcar{1:r}(0),Y\pcar{1:r}(0),A\pcar{1:r,1:r}(0)\right)}.
\end{align*}

This proves the theorem when $t = 0$.\\

\textbf{Inductive Claim I:} For each $n > r$, define
\begin{equation}
\label{eqn::altmu}
\alt{\mu}_t \defeq \frac{1}{n}\sum_{j=1}^n \delta_{Z^n_{1,j}[t],Y^n_{1,j}[t],A^n_{1j}(t)}\te{ and } \alt{\mu}_t \defeq \law\left(Z\pcar{1,2}[t],Y\pcar{1,2}[t],A\pcar{12}(t)\middle|Z\pcar{1}[t],Y\pcar{1}[t]\right).
\end{equation} 
Then $(Z^n_1[t],Y^n_1[t],\alt{\mu}^n_t)\Rightarrow (Z\pcar{1}[t],Y\pcar{1}[t],\alt{\mu}_t)$.\\

\begin{proof}[Proof of Inductive Claim 1]
The claim follows along the lines of  \cite[Lemma 7.9]{GanSpiSus25} when $q = m = 1$. We omit the full proof, and focus on the differences between the two proofs.

\ind The proof of \cite[Lemma 7.9]{GanSpiSus25} proceeds by applying a conditional propagation of chaos result, specifically, Proposition 5.5 of \cite{GanSpiSus25}, which we restate in Proposition \ref{prop::condconvres} for convenience. We take the same approach: define $\Xmc \defeq \mf_t(\R^d)\times\mf_t([1:q])$, $\Ymc \defeq \mf_t(\R^d)\times\mf_t([1:q])\times [m]$, $X \defeq (Z\pcar{1}[t],Y\pcar{1}[t])$, $Y \defeq (Z\pcar{2}[t],Y\pcar{2}[t],A\pcar{12}[t])$, $X^n \defeq (Z^n_1[t],Y^n_1[t])$, and $Y^n_i = (Z^n_i[t],Y^n_i[t],A^n_{1i}[t])$ for all $n \in \N, i \in [1:n]$. The proof that these random elements satisfy the conditions of Proposition \ref{prop::condconvres} is nearly identical to the analogous statement in the proof of Lemma 7.9 of \cite{GanSpiSus25} with one exception. To establish the continuous dependence of the law $\eta \defeq \law(X,Y|X)$ on $X$ requires a little extra work in our setting.

\ind It suffices to show that for each $a \in [m]$, the conditional probability $\PP(A\pcar{12}(t) = a|\Gmc^2_t)$ is an a.s. continuous function of $(Z\pcar{1,2}[t],Y\pcar{1,2}[t])$. We note by Corollary \ref{coro::morestruct}, \eqref{eqn::Bevolve}, and \eqref{eqn::Atprob} that
\begin{align*}
&\PP\left(A\pcar{12}(t) = a\middle|\Gmc^2_t\right) = B_{t,a}(Z\pcar{1,2}[t],Y\pcar{1,2}[t]) \\
&\quad = \sum_{\substack{\bm{a} \in \mf_t([m])\\ \bm{a}(t) = a}} B_0(Z\pcar{1,2}(0),Y\pcar{1,2}(0),\bm{a}(0))\prod_{s=1}^t B(Z\pcar{1,2}(s),Y\pcar{1,2}(s),\bm{a}[s-1:s])\te{ a.s..}
\end{align*}
$f_a \defeq B_{t,a}$ is therefore a.s. continuous for each $a \in [m]$. Then, just as in the proof in \cite{GanSpiSus25}, we may show that $\eta$ continuously depends on $X$ by applying \cite[Lemma A.1]{GanSpiSus25}. And the result follows from an application of Proposition \ref{prop::condconvres} (or equivalently, \cite[Proposition 5.5]{GanSpiSus25}) as is done to prove Lemma 7.9 of \cite{GanSpiSus25}.
\end{proof}

\textbf{Inductive Claim II:} 
\begin{align*}
&\left(Z^n_1[t],Y^n_1[t],\left(\frac{1}{n}\sum_{j=1}^n Z^n_j(t)\indic{A^n_{1j}(t) = k},\frac{1}{n}\sum_{j=1}^n\indic{A_{1j}(t) = k}\right)_{k=1}^m\right) \\
&\quad \Rightarrow \left(Z\pcar{1}[t],Y\pcar{1}[t],\left(\ex{Z\pcar{2}[t]\indic{A\pcar{12}(t)=k}\middle|Z\pcar{1}[t],Y\pcar{1}[t]},\ex{\indic{A\pcar{12}(t)=k}\middle|Z\pcar{1}[t],Y\pcar{1}[t]}\right)_{k=1}^n\right).
\end{align*}

\begin{proof}[Proof of Inductive Claim II]
This proof is identical to the first part of the proof of \cite[Lemma 7.10]{GanSpiSus25} up to notational changes. 
\end{proof}

\textbf{Inductive Claim III:} 
\begin{align*}
&\left(Z^n_1[t],Y^n_1[t],L^{n,1:m}_1(t)\right) \Rightarrow \left(Z\pcar{1}[t],Y\pcar{1}[t],L\pcar{1}_{1:m}(t)\right).
\end{align*}
\begin{proof}[Proof of Inductive Claim III]
In \cite{GanSpiSus25}, this claim followed from Claim II using a simple argument that the denominator of $L\pcar{1}(t)$ is a.s. positive, so that the result follows from the continuous mapping theorem. In our setting, it is possible for the denominator of $L\pcar{1}_k(t)$ to be zero (in which case the numerator is also $0$, so $L\pcar{1}_k(t) = 0$ via our convention that $0/0=0$). So, although we still prove this claim by applying the continuous mapping theorem, the argument is more involved and requires the use of two technical lemmas (Lemmas \ref{lem::vanishbound} and \ref{lem::posdenom}) from Section \ref{ssec::tech}.

\ind In Section \ref{ssec::tech}, Lemma \ref{lem::vanishbound} establishes the existence of a set $\Emc\in \borel\left(\mf_t(\R^d\times[1:q])\right)$ such that
\begin{equation}
\label{eq::Emczero}
\PP\left(A\pcar{12}=k\middle|\Gmc^1_t\right)\indic{(Z\pcar{1}[t],Y\pcar{1}[t]) \in \Emc} = 0 \te{ a.s.,}
\end{equation}
\[\PP\left(\PP\left(A\pcar{12}=k\middle|\Gmc^1_t\right)= 0,(Z\pcar{1}[t],Y\pcar{1}[t]) \in \Emc^c\right) = 0,\]
and such that the boundary $\partial \Emc$ is $(Z\pcar{1}[t],Y\pcar{1}[t])$-null. In simple terms, $\Emc$ is the maximal set such that \eqref{eq::Emczero} holds (up to null sets). For each $k$, let $f_k, g_k: \mf_t(\R^d)^2\times\mf_t ([1:q])^2\times [m] \to \R$ be defined by $f_k(x_{1,2},a) \defeq z_1\indic{a=k}$ and $g_k(x_{1,2},a) \defeq \indic{a=k}$ for all $x_i = (z_i,y_i) \in \mf_t(\R^d)\times \mf_t([1:q])$ and $a \in [m]$. Recall the definitions of $\alt{\mu}^n_t$ and $\alt{\mu}_t$ from \eqref{eqn::altmu}. We simplify the notation above by introducing the following quantities for each $n \in \N$ and $k\in [1:m]$:
\begin{align*}
N^{n,k} \defeq \langle\alt{\mu}^n_t,f_k\rangle, \quad N\pcar{k} \defeq \langle \alt{\mu}_t,f_k\rangle,\quad
D^{n,k} \defeq \langle \alt{\mu}^n_t,g_k\rangle, \quad
D\pcar{k} \defeq \langle \alt{\mu}_t,g_k\rangle.
\end{align*}
In addition, consider the following function $\psi: \mf_t(\R^d\times[1:q]) \times \R \to \R$ defined by
\[\psi(x,d) = \begin{cases}
d &\te{ if } x \notin \Emc,\\
1&\te{ otherwise,}
\end{cases} = d\indic{x \notin \Emc} + \indic{x \in \Emc}.\]

By Inductive Claim II, if $X^n = (Z^n_1[t],Y^n_1[t])$ and $X = (Z\pcar{1}[t],Y\pcar{1}[t])$,
\begin{equation}
\label{eqn::ZsZAZAconvergev2}
(X^n,(N^{n,k},D^{n,k})_{k=1}^m) \Rightarrow (X,(N\pcar{k},D\pcar{k})_{k=1}^m).
\end{equation}
Second, by Lemma \ref{lem::posdenom} of Section \ref{ssec::tech} and by the definition of $D^{n,k}$,
\begin{equation}
\label{eqn::Dnkzero}
D^{n,k}\indic{X^n\in \Emc} = \langle \alt{\mu}^n_t,f_k\rangle\indic{X^n\in \Emc}= d_{A^{n,k}(t)}(1)\indic{(Z^n_1[t],Y^n_1[t])\in \Emc} = 0\te{ a.s..}
\end{equation}
Third, by definition of $\Emc$,
\begin{equation}
\label{eqn::Dknonzero}
\PP\left(D\pcar{k}=0,X^n \in \Emc^c\right) = \PP\left(\PP\left(A\pcar{12}(t) = k\middle|\Gmc^1_t\right)=0,(Z\pcar{1}[t],Y\pcar{1}[t]) \in \Emc^c\right) = 0.
\end{equation}
Fourth, given that $N^{n,k} = 0$ when $D^{n,k} = 0$ and recalling our convention that $0/0 = 0$, it follows that
\begin{equation}
\label{eqn::NnDnNnpsiDn}
\frac{N^{n,k}}{\psi(X^n,D^{n,k})} = \frac{N^{n,k}}{D^{n,k}} \te{ a.s. for }k \in [1:m].
\end{equation}

Similarly, \eqref{eq::Emczero} and \eqref{eqn::Dknonzero} imply
\[\psi(X,D\pcar{k}) = \begin{cases}
D\pcar{k} &\te{ if } D\pcar{k} > 0,\\
1 & \te{ if } D\pcar{k}=0.
\end{cases}.\]
Therefore, because $N\pcar{k} =0$ when $D\pcar{k} =0$, our $0/0=0$ convention implies the following:
\begin{equation}
\label{eqn::NDNpsiD}
\frac{N\pcar{k}}{\psi(X,D\pcar{k})} = \frac{N\pcar{k}}{D\pcar{k}}\te{ a.s. for all }k.
\end{equation}
Lastly, by Lemma \ref{lem::vanishbound}, $\PP\left(X \in \partial \Emc\right) = 0$ which makes the following mapping $(X,(N\pcar{k},D\pcar{k})_{k=1}^m)$-a.s. continuous:
\[\mf_t(\R^d,[1:q])\times \R^m\times \R^m \ni (x,(n_k,d_k)_{k=1}^m) \mapsto \left(x,\left(\frac{n_k}{\psi(x,d_k)}\right)_{k=1}^m\right) \in \mf_t(\R^d,[1:q])\times\R^m.\]
Combining \eqref{eqn::ZsZAZAconvergev2}-\eqref{eqn::NDNpsiD} with the a.s. continuity of the above mapping, the continuous function theorem implies
\begin{align*}
\left(X^n,\left(\frac{N^{n,k}}{D^{n,k}}\right)_{k=1}^m\right) = \left(X^n,\left(\frac{N^{n,k}}{\psi(X^n,D^{n,k})}\right)_{k=1}^m\right) \Rightarrow \left(X,\left(\frac{N\pcar{k}}{\psi(X,D\pcar{k})}\right)_{k=1}^m\right) = \left(X,\left(\frac{N\pcar{k}}{D\pcar{k}}\right)_{k=1}^m\right) \te{ a.s..}
\end{align*}
Hence, 
\begin{align*}
&\left(Z^n_1[t],Y^n_1[t],\left(\frac{\frac{1}{n}\sum_{j=1}^n Z^n_j(t)\indic{A^n_{1j}(t) = k}}{\frac{1}{n}\sum_{j=1}^n \indic{A^n_{1j}(t)=k}}\right)_{k=1}^m\right)\\
&\ind\ind \Rightarrow \left(Z\pcar{1}[t],Y\pcar{1}[t],\left(\frac{\ex{Z\pcar{2}[t]\indic{A\pcar{12}(t)=k}\middle|Z\pcar{1}[t],Y\pcar{1}[t]}}{\ex{\indic{A\pcar{12}(t)=k}\middle|Z\pcar{1}[t],Y\pcar{1}[t]}}\right)_{k=1}^m\right).
\end{align*}

Recall the definition of the function $B_{s,a}: (\mf_t(\R^d\times[1:q]))^2\to [0,1]$, $s \in [t], a \in [m]$. Then,
\begin{align*}
N\pcar{k} &= \ex{Z\pcar{2}\indic{A\pcar{12}(t)=k}\middle|\Gmc^1_t} \\
&= \ex{Z\pcar{2}\ex{\indic{A\pcar{12}(t)=k}\middle|\Gmc^2_t}\middle|\Gmc^1_t} \\
&= \ex{Z\pcar{2}B_{t,k}(Z\pcar{1,2}[t],Y\pcar{1,2}[t])\middle|Z\pcar{1}[t],Y\pcar{1}[t]}.
\end{align*}
Likewise,
\[D\pcar{k} = \ex{B_{t,k}(Z\pcar{1,2}[t],Y\pcar{1,2}[t])\middle|Z\pcar{1}[t],Y\pcar{1}[t]}.\]
So for $k \in [1:m]$, $\frac{N\pcar{k}}{D\pcar{k}} = L\pcar{1}_k$ a.s.. Hence,
\[\left(Z^n_1[t],Y^n_1[t],L^{n,1:m}_1(t)\right) \Rightarrow \left(Z\pcar{1}[t],Y\pcar{1}[t],L\pcar{1}_{1:m}(t)\right).\]
\end{proof}

\textbf{Inductive Claim IV:} 
\[\left(Z^n_{1:r}[t],Y^n_{1:r}[t],A^n_{1:r,1:r}[t],L^{n,1:m}_{1:r}(t)\right) \Rightarrow \left(Z\pcar{1:r}[t],Y\pcar{1:r}[t],A\pcar{1:r,1:r}[t],L\pcar{1:r}_{1:m}(t)\right).\]
\begin{proof}[Proof of Inductive Claim IV]
By joint exchangeability of $(Z^n,Y^n,A^n)$, the distribution of $(Z^n_i[t],Y^n_i[t],L^{n,1:m}_i(t))$ does not depend on the choice of $i$. Likewise, as a consequence of Definition \ref{defn::AZlim}, the distribution of $(Z\pcar{i}[t],Y\pcar{i}[t],L\pcar{i}_{1:m}(t))$ does not depend on $i$. Therefore, for each $i \in [1:r]$, Inductive Claim III implies that
\begin{equation}
\label{eqn::ZYLiconv}
\left(Z^n_i[t],Y^n_i[t],L^{n,1:m}_i(t)\right) \Rightarrow \left(Z\pcar{i}[t],Y\pcar{i}[t],L\pcar{i}_{1:m}(t)\right).
\end{equation}

Then we prove the result using the conditional Slutzky's theorem as defined in \cite[Lemma A.2]{GanSpiSus25}. In the notation of that theorem, let $\Xmc = \mf_t(\R^d)\times\mf_t([1:q])\times \mf_t([m])^{r\times r}$. Let $\Ymc = (\R^d)^m$. Define $k = r+1$. For $n \in \N$ and $i \leq r$ let $X^n_i = (Z^n_i[t],Y^n_i[t])$. Let $X^n_{r+1} = A^n_{1:r,1:r}[t]$. For $n \in \N$ and $i \leq r$, define $Y^n_i = L^{n,1:m}_i(t)$. Set $Y^n_{r+1} \defeq \bm{0}$. Lastly, define $X_i = (Z\pcar{i}[t],Y\pcar{i}[t])$, $Y_i = L\pcar{i}_{1:m}(t)$ for $i \leq r$ and $X_{r+1} \defeq A\pcar{1:r,1:r}[t]$ and $Y_{r+1} \defeq \bm{0}$. Then by \eqref{eqn::ZYLiconv}, $(X^n_i,Y^n_i)\Rightarrow (X_i,Y_i)$ for $i\leq r$. By assumption, $(Z^n_{1:r}[t],Y^n_{1:r}[t],A^n_{1:r,1:r}[t])\Rightarrow (Z\pcar{1:r}[t],Y\pcar{1:r}[t],A\pcar{1:r,1:r}[t])$ which implies $(X^n_{r+1},Y^n_{r+1}) \Rightarrow (X_{r+1},Y_{r+1})$ and $X^n_{1:r+1} \Rightarrow X_{1:r+1}$. Lastly, the random vectors $L\pcar{i}_{1:m}(t)$ are $\sigma(Z\pcar{i}[t],Y\pcar{i}[t])$-measurable, so for $i \leq r$, there exists a measurable function $\phi_i$ such that $\phi_i(X_i) = Y_i$. For $i = r+1$, we may simply set $\phi_{r+1} \equiv \bm{0}$ which also yields $\phi_{r+1}(X_{r+1}) = \bm{0} = Y_{r+1}$. This completes the verification of the conditions of \cite[Lemma A.2]{GanSpiSus25}, so
\begin{align*}
\left(Z^n_{1:r}[t],Y^n_{1:r}[t],A^n_{1:r,1:r}[t],L^{n,1:m}_{1:r}(t),\bm{0}\right) &= (X^n_{1:r+1},Y^n_{1:r+1}) \\
&\Rightarrow (X_{1:r+1},Y_{1:r+1}) \\
&= \left(Z\pcar{1:r}[t],Y\pcar{1:r}[t],A\pcar{1:r}[t],L\pcar{1:r}_{1:m}(t),\bm{0}\right).
\end{align*}
This completes the proof of the lemma.
\end{proof}

\textbf{Inductive Claim V:} 
\[\left(Z^n_{1:r}[t+1],Y^n_{1:r}[t],A^n_{1:r,1:r}[t]\right)\Rightarrow \left(Z\pcar{1:r}[t+1],Y\pcar{1:r}[t],A\pcar{1:r,1:r}[t]\right).\]
\begin{proof}[Proof of Inductive Claim V]
$\xi_{1:r}(t)$ is independent of $(Z^n_{1:r}[t],Y^n_{1:r}[t],A^n_{1:r,1:r}[t],L^{n,1:m}_{1:r}(t))$. It is also independent of $(Z\pcar{1:r}[t],Y\pcar{1:r}[t],A\pcar{1:r}[t],L\pcar{1:r}_{1:m}(t))$. It follows that
\[\left(Z^n_{1:r}[t],Y^n_{1:r}[t],A^n_{1:r,1:r}[t],L^{n,1:m}_{1:r}(t),\xi_{1:r}(t)\right) \Rightarrow  \left(Z\pcar{1:r}[t],Y\pcar{1:r}[t],A\pcar{1:r}[t],L\pcar{1:r}_{1:m}(t),\xi_{1:r}(t)\right).\]
The result follows by the continuous mapping theorem.
\end{proof}

\textbf{Inductive Claim VI:} 
\[\left(Z^n_{1:r}[t+1],Y^n_{1:r}[t+1],A^n_{1:r,1:r}[t]\right)\Rightarrow \left(Z\pcar{1:r}[t+1],Y\pcar{1:r}[t+1],A\pcar{1:r,1:r}[t]\right).\]
\begin{proof}[Proof of Inductive Claim VI]

\ind We can apply  \cite[Lemma A.3]{GanSpiSus25} described above. Let $\Xmc = (\mf_{t+1}(\R^d))^r\times (\mf_t([1:q])^r\times \left(\mf_t\left([m]\right)\right)^{r\times r}$. For each $n$ let $X_n \defeq (Z^n_{1:r}[t+1],Y^n_{1:r}[t],A^n_{1:r,1:r}[t])$ be an $\Xmc$-random element. In addition, let $X \defeq (Z\pcar{1:r}[t+1],Y\pcar{1:r}[t],A\pcar{1:r,1:r}[t])$ be an $\Xmc$-random element and note that $\sigma(X) = \Gmc^{Y,r}_{t+1}$. Then by the above display, $X_n \Rightarrow X$, and $\sigma(X_n) = \sigma\left(Z^n_{1:r}[t+1],Y^n_{1:r}[t],A^n_{1:r,1:r}[t]\right)\subset \Fmc^{Y,n}_{t+1}$. For each $i \in [1:r]$, let $B^n_i = Y^n_i(t+1)$ be a $[1:q]$-random variable. Then by Assumption \ref{assu::cond}(b), $\{B^n_i\}_{i=1}^r$ is a conditionally independent sequence of random variables given $\Fmc^{Y,n}_{t+1}$. Given that $\sigma(X_n)\subset \Fmc^{Y,n}_{t+1}$ and for each $i$ and $y$,
\[\PP(B^n_i=y|\Fmc^{Y,n}_{t+1}) = P(Z^n_{i}(t+1),Y^n_i(t),y)\]
is $\sigma(X_n)$-measurable,  \cite[Lemma A.5]{GanSpiSus25} above implies that that $\{B^n_i\}_{i \in [1:r]}$ is a conditionally independent sequence of random variables given $X_n$ for each $n$. Moreover, 
\[\phi^a_i(X_n) \defeq P\left(Z^n_{i}(t+1),Y^n_i(t),a\right),\]
is bounded and a.e. continuous by definition. This implies that $\phi^a_i$ is also $X$-a.s. continuous by Lemma \ref{lem::abscont}. Therefore all the properties of  \cite[Lemma A.3]{GanSpiSus25} are satisfied, so
\[(X^n,B^n_{1:r}) \Rightarrow (X,B\pcar{1:r}),\]
where $\{B\pcar{i}\}_{i=1}^{r}$ are mutually conditionally independent given $\sigma(X) = \Gmc^{Y,r}_{t+1}$ and 
\[\PP\left(B\pcar{i}=a|X\right) = \PP\left(B\pcar{i}=a\middle|\Gmc^{A,r}_{t+1}\right) = P\left(Z\pcar{i}(t+1),Y\pcar{i}(t),a\right).\]

By Definition \ref{defn::AZlim} and Proposition \ref{prop::welldef}, it follows that
\begin{align*}
\left(Z\pcar{1:r}[t+1],Y\pcar{1:r}[t],A\pcar{1:r,1:r}[t],B\pcar{1:r}\right)\deq \left(Z\pcar{1:r}[t+1],Y\pcar{1:r}[t+1],A\pcar{1:r,1:r}[t]\right).
\end{align*}
This completes the proof.
\end{proof}

\textbf{Final Inductive Claim:}
\[\left(Z^n_{1:r}[t+1],Y^n_{1:r}[t+1],A^n_{1:r,1:r}[t+1]\right)\Rightarrow \left(Z\pcar{1:r}[t+1],Y\pcar{1:r}[t+1],A\pcar{1:r,1:r}[t+1]\right).\]
\begin{proof}[Proof of the Final Inductive Claim]
Again, this argument is the same as the proof of Property A(c) holding at time $t+1$ given Lemma 7.12 found on page 44 of \cite{GanSpiSus25}.
\end{proof}

Combining the seven inductive claims, we get that $(Z^n_{1:r}[t],Y^n_{1:r}[t],A^n_{1:r,1:r}[t]) \Rightarrow (Z\pcar{1:r}[t],Y\pcar{1:r}[t],A\pcar{1:r,1:r}[t])$ implies $(Z^n_{1:r}[t+1],Y^n_{1:r}[t+1],A^n_{1:r,1:r}[t+1]) \Rightarrow (Z\pcar{1:r}[t+1],Y\pcar{1:r}[t+1],A\pcar{1:r,1:r}[t+1])$. Then Theorem \ref{thm::Sampleconv} follows by induction.

\subsection{Technical Lemmas}
\label{ssec::tech}
The assumptions of \cite{GanSpiSus25} ensured that all nodes in the network interacted with $\theta(n)$ other nodes, which ensured that the denominator of the averaging term in the interacting particle system never approached 0 as $n \to \infty$. This was a reasonable assumption to make in the context that there was only one type of interaction in the system. In our paper, we allow for $m \geq 1$, so it is natural that we must consider the case that some nodes may have no interactions of a given type $k \in [1:m]$ at time $t$. For any such node $i$, $L^{n,k}_i(t)$ would have a denominator of zero, and in the limit, we expect the denominator of $L\pcar{i}_k(t)$ to vanish with some positive probability. 

\ind In this section, we provide a few technical lemmas that help control this case. We begin with a Lemma \ref{lem::support} showing that the information about two particles in the $n$-particle system and in the limiting mean-field system have the same support.

\begin{lemma}
\label{lem::support}
For every $n \geq 2$, the random quantities $V[t] \defeq (Z\pcar{1,2}[t],Y\pcar{1,2}[t],A\pcar{12}[t])$ and $V^n[t]\defeq (Z^n_{1,2}[t],Y^n_{1,2}[t],A^n_{12}[t])$ are both supported on the set
\begin{align}
S_t = &\Bigg\{(\bm{z}_{1,2},\bm{y}_{1,2},\bm{a}) \in \mf_t(\R^d)^2\times \mf_t ([1:q])^2\times\mf_t([m]): \nonumber\\
&\ind\ind B_0(z_{1,2}(0),y_{1,2}(0),a(0))\left(\prod_{i=1}^2 P_0(z_i(0),y_i(0))\right)\nonumber\\
&\ind\ind \left(\prod_{s=1}^t B(z_{1,2}(s),y_{1,2}(s),a[s-1:s])\prod_{i=1}^2 P(z_i(s),y_i[s-1:s])\right) > 0\Bigg\}.\label{eqn:S_t}
\end{align}

Furthermore, if we define $\ell_t$ to be the product of the Lebesgue measure on $\mf_t((\R^d))^2$ and the counting measure on the finite set $\mf_t([1:q])^2\times\mf_t([m])$, then the respective laws of $V[t]$ and $V^n[t]$ are both absolutely continuous with respect to $\ell_t$, and their densities are strictly positive on $S_t$.
\end{lemma}
\begin{proof}
For any $i \in \{1,2\}$ and $n \geq 2$, Assumption \ref{assu::init}(a) implies that both $Z\pcar{1,2}(0)$ and $Z^n_{1,2}(0)$ are absolutely continuous and of full measure. Label their respective density functions by $f_{z,0}$ and $f^n_{z,0}$ and note that both density functions are strictly positive. Then by Assumption \ref{assu::init}(b) and (c), $V(0)$ and $V^n(0)$ are both absolutely continuous with respect to $\ell_0$ with respective densities:
\begin{align}
f_0(v) &= f_{z,0}(z_{1,2})B_0(z_{1,2},y_{1,2},a)\prod_{i=1}^2 P_0(z_i,y_i),\nonumber\\
f^n_0(v) &= f^n_{z,0}(z_{1,2})B_0(z_{1,2},y_{1,2},a)\prod_{i=1}^2 P_0(z_i,y_i)\label{eqn::V0dense}
\end{align}
for $v=(z_{1,2},y_{1,2},a) \in (\R^d)^2\times ([1:q])^2\times[m]$.

\ind Now assume that for some $s > 0$, $V[s-1]$ and $V^n[s-1]$ are both absolutely continuous with respect to $\ell_{s-1}$ and with respective density functions $f_{s-1}$ and $f^n_{s-1}$. By \eqref{eqn::Zevolve}-\eqref{eqn::mu2def} and \eqref{eqn::Znevolve}-\eqref{eqn::dAnt}, $(Z\pcar{1,2}(s) - \xi_{1,2}(s))$ and $(Z^n_{1,2}(s)-\xi_{1,2}(s))$ are respectively $\Tmc^2_{s-1}$ and $\Fmc^n_{s-1}$-measurable random vectors in $(\R^d)^2$. Therefore, there exists a measurable function $\phi_s: \mf_t(\R^d)^2\times\mf_t([1:q])^2\times\mf_t([m]) \to (\R^d)^2$ such that,
\[\phi_s(V[s-1]) = Z\pcar{1,2}(s)-\xi_{1,2}(s)\te{ a.s..}\]
Recall from Assumption \ref{assu::cond} and the paragraph after \eqref{eqn::mu2def} that $\xi_{1,2}(s)$ is absolutely continuous, of full support and independent of both $\Tmc^2_{s-1} \defeq \sigma(V[s-1])$ and $\Fmc^n_{s-1} \defeq \sigma(Z^n[s-1],Y^n[s-1],A^n[s-1])$. Let $f_{\xi,s}$ denote the density of $\xi_{1,2}(s)$. Then the density of $(Z\pcar{1,2}(s),V[s-1])$ with respect to the measure $\lambda\times\ell_{s-1}$ (where $\lambda$ is the Lebesgue measure on $(\R^d)^2$) is given by
\begin{align*}
f_{z,s}(z'_{1,2},v[s-1]) &= f_{s-1}(v[s-1])f_{\xi,s}(z'_{1,2} - \phi_s(v[s-1])),
\end{align*}
for $v[s-1] \in \mf_{s-1}(\R^d)^2\times \mf_{s-1}([1:q])^2\times \mf_{s-1}([m])$ and $z'_{1,2} \in (\R^d)^2$. Finding the density of $v^n[s-1]$ is slightly more complex as $Z^n_{1,2}(s) - \xi^{1,2}(s)$ is not measurable with respect to $\sigma(V^n[s-1])$. However, this density can still by found by taking an appropriate conditional expectation:
\begin{equation}
\label{eqn::fnzsdef}
f^n_{z,s}(z'_{1,2},v[s-1]) = \ex{f^n_{s-1}(v[s-1])f_{\xi,s}(z'_{1,2} - (Z^n_{1,2}(s)-\xi^{1,2}(s)))\middle|V^n[s-1] = v[s-1]},
\end{equation}
where once more $v[s-1] \in \mf_{s-1}(\R^d)^2\times \mf_{s-1}([1:q])^2\times \mf_{s-1}([m])$ and $z'_{1,2} \in (\R^d)^2$. See Appendix \ref{sec::fnzsjust} for a derivation of the above density function.

\ind Applying \eqref{eqn::peqn}, \eqref{eqn::Anevolve}, \eqref{eqn::Yevolve} and \eqref{eqn::Aevolve}, we find that the respective densities $f_s$ and $f^n_s$ are given by
\begin{align}
f_s(v[s]) &= f_{s-1}(v[s-1])f_{\xi,s}(z_{1,2}(s) - \phi_s(v[s-1]))B(z_{1,2}(s),y_{1,2}(s),a[s-1:s])\prod_{i=1}^2 P(z_i(s),y_i[s-1:s]),\nonumber\\
f^n_s(v[s]) &= f^n_{s-1}(v[s-1])\ex{f_{\xi,s}(z_{1,2}(s) - (Z^n_{1,2}(s)-\xi^{1,2}(s)))\middle|V^n[s-1] = v[s-1]}\nonumber\\
&\ind\ind B(z_{1,2}(s),y_{1,2}(s),a[s-1:s])\prod_{i=1}^2 P(z_i(s),y_i[s-1:s]),\label{eqn::Vsdenserec}
\end{align}
for $v[s] = (z_{1,2}[s],y_{1,2}[s],a[s]) \in \mf_{s}(\R^d)^2\times \mf_{s}([1:q])^2\times \mf_{s}([m])$.

\ind Applying \eqref{eqn::V0dense} and \eqref{eqn::Vsdenserec}, we get an explicit expression for the respective densities of $V[t]$ and $V^n[t]$:
\begin{align}
f_t(v) &= f_{z,0}(z_{1,2}(0))\left(\prod_{s=1}^t f_{\xi,s}\left(z_{1,2}(s)-\phi_s(v[s-1])\right)\right) B_0\left(z_{1,2}(0),y_{1,2}(0),a(0)\right)\left(\prod_{i=1}^2 P_0(z_i(0),y_i(0))\right)\nonumber\\
&\ind\ind \left(\prod_{s=1}^t B(z_{1,2}(s),y_{1,2}(s),a[s-1:s])\prod_{i=1}^2 P(z_i(s),y_i[s-1:s])\right)\label{eqn::Vdense}\\
f^n_t(v) &= f^n_{z,0}(z_{1,2}(0))\left(\prod_{s=1}^t \ex{f_{\xi,s}\left(z_{1,2}(s)-\left(Z^n_{1,2}(s)-\xi_{1,2}(s)\right)\right)\middle|V^n[s-1] = v[s-1]}\right) \nonumber\\
&\ind\ind B_0\left(z_{1,2}(0),y_{1,2}(0),a(0)\right)\left(\prod_{i=1}^2 P_0(z_i(0),y_i(0))\right)\nonumber\\
&\ind\ind \left(\prod_{s=1}^t B(z_{1,2}(s),y_{1,2}(s),a[s-1:s])\prod_{i=1}^2 P(z_i(s),y_i[s-1:s])\right)\label{eqn::Vndense},
\end{align}
for $v = (z_{1,2},y_{1,2},a) \in \mf_t(\R^d)^2\times\mf_t([1:q])^2\times \mf_t([m])$. This completes the proof of the second claim of the lemma. 

\ind To prove the first claim, note that $V[t]$ and $V^n[t]$ have the same supports as the respective functions $f_t$ and $f^n_t$. Moreover, by assumption, $f^n_{z,0}, f_{z,0}$ and $f_{\xi,s}$ are strictly positive functions for all $n \geq 2$ and $s \in [t]$. It therefore follows from \eqref{eqn::Vdense} and \eqref{eqn::Vndense} that the supports of both $V[t]$ and $V^n[t]$ are given by the set $S_t$ from \eqref{eqn:S_t} and that the densities of the two random elements are strictly positive on $S_t$.
\end{proof}

Next, we apply the above lemma to show that we do not have any circumstances in which the denominator of $L^{n,k}_i(t)$ is non-zero, but the denominator of its limit $L\pcar{i}_k(t)$ is zero (or vice versa). Instead, the events on which the denominator of the averaging terms vanish coincide in the $n$-particle system and the mean-field process.

\begin{lemma}
\label{lem::posdenom}
For each $k \in [1:m]$ and $\Emc \in \borel\left(\mf_t(\R^d\times[1:q])\right)$ such that
\[\PP\left(A\pcar{12}(t)=k\middle|Z\pcar{1}[t],Y\pcar{1}[t]\right)\indic{(Z\pcar{1}[t],Y\pcar{1}[t]) \in \Emc} = 0\te{ a.s.,}\]
$d_{A^{n,k}(t)}(1)\indic{(Z^n_1[t],Y^n_1[t]) \in \Emc}=0$ a.s. for $n$ sufficiently large.
\end{lemma}
\begin{proof}
Fix $k\in [1:m]$ and fix any $\Emc \in \borel(\mf_t(\R^d)\times\mf_t([1:q])$ such that 
\begin{equation}
\label{eq::AZYe0}
\PP(A\pcar{12}(t)=k|Z\pcar{1}[t],Y\pcar{1}[t])\indic{(Z\pcar{1}[t],Y\pcar{1}[t]) \in \Emc} = 0\te{ a.s..}
\end{equation}
Let $\Smc$ be the support of the random variables $(Z\pcar{1,2}[t],Y\pcar{1,2}[t],A\pcar{12}[t])$, and define
\[\Smc_{\Emc} \defeq \left\{(z_{1,2},y_{1,2},a) \in \mf_t(\R^d)^2\times \mf_t([1:q])^2\times \mf_t([m]): a(t) = k, (z_1,y_1) \in \Emc\right\}.\]
Then \eqref{eq::AZYe0} holds if and only if
\[\PP\left(A\pcar{12}(t) = k,(Z\pcar{1}[t],Y\pcar{1}[t]) \in \Emc\right) = \PP\left((Z\pcar{1,2}[t],Y\pcar{1,2}[t],A\pcar{12}[t]) \in \Smc_{\Emc}\right) = 0.\]
By Lemma \ref{lem::support}, $(Z\pcar{1,2}[t],Y\pcar{1,2}[t],A\pcar{12}[t])$ and $(Z^n_{1,2}[t],Y^n_{1,2}[t],A^n_{12}[t])$ have the same support for $n \geq 2$, so 
\begin{align*}
\PP\left(A^n_{12}(t) = k,(Z^n_{1},Y^n_{1}[t]) \in \Emc\right) &= \PP\left((Z^n_{1,2}[t],Y^n_{1,2}[t],A^n_{12}[t]) \in \Smc_{\Emc}\right) \nonumber\\&=\PP\left((Z\pcar{1,2}[t],Y\pcar{1,2}[t],A\pcar{12}[t]) \in \Smc_{\Emc}\right) = 0.
\end{align*}

By ex-changeability, this implies
\[\PP\left(A^n_{1j}(t) = k,(Z^n_{1},Y^n_{1}[t]) \in \Emc\right) = 0\]
for all $j \in [2:n]$. Moreover, $A^n_{11}(t) = 0$ by definition. Thus,
\begin{align*}
d_{A^{n,k}(t)}(1)\indic{(Z^n_1[t],Y^n_1[t]) \in \Emc} = \sum_{j=1}^n \indic{A^n_{1j}(t) = k,(Z^n_1[t],Y^n_1[t]) \in \Emc} = 0\te{ a.s..}
\end{align*}
\end{proof}

Our results require both the numerator and denominator of each $L^{n,k}_i$ to converge when scaled by $1/n$. At its foundation, our argument is a law of large numbers result. Therefore, it is not sufficient for the denominator of $L^{n,k}_i$ to be non-zero if and only if $L\pcar{i}(t)$ has a non-zero denominator. In fact, we need to show that the denominator is either equal to $0$ or of order $\theta(n)$. More specifically, the boundary separating these events must itself be a null event. This lemma is used to establish that claim. Specifically, the following lemma is used to prove that division of the numerator and denominator of $L\pcar{i}_k(t)$ (with the $0/0=0$ convention) is an a.s. continuous operation. Therefore, if we know that the numerator and denominator of $L^{n,k}_i(t)$ scaled by $1/n$ jointly weakly converge to the numerator and denominator of $L\pcar{i}_k(t)$, then by the continuous mapping theorem, $L^{n,k}_i(t)$ converges weakly to $L\pcar{i}_k(t)$.
\begin{lemma}
\label{lem::vanishbound}
For every $k \in [1:m]$, there exists a set $\Emc \subseteq \borel\left(\mf_t(\R^d\times[1:q])\right)$ such that
\begin{equation}
\label{eqn::AGE}
\PP\left(A\pcar{12}(t) = k\middle|\Gmc^1_t\right)\indic{(Z\pcar{1}[t],Y\pcar{1}[t]) \in \Emc} = 0 \te{ a.s.}
\end{equation}
and 
\begin{equation}
\label{eqn::AGEc}
\PP\left(\PP\left(A\pcar{12}(t) = k\middle|\Gmc^1_t\right) = 0, (Z\pcar{1}[t],Y\pcar{1}[t]) \notin \Emc\right) = 0.
\end{equation}
Moreover, 
\[\PP\left((Z\pcar{1}[t],Y\pcar{1}[t]) \in \partial \Emc\right) = 0.\]
\end{lemma}
\begin{proof}
Recall the sets $\Dmc_0(\cdot)$ and $\Dmc(\cdot)$ as defined in Assumption \ref{assu::Bbds}. For a fixed $t \geq 0$ and $0 < s \leq t$, consider the following sets in $\borel(\mf_t(\R^d\times [1:q]))$ parameterized by $(\bm{y}',\bm{a}') \in \mf_t([1:q])\times\mf_t([m])$:
\begin{align}
\bar{\Dmc}_{0,t}\left(\bm{y}',\bm{a}\right) &\defeq \left\{(\bm{z},\bm{y})\in \mf_t(\R^d\times [1:q]): z(0) \in \Dmc_0\left(y(0),y'(0),a(0)\right)\right\},\nonumber\\
\bar{\Dmc}_{s,t}\left(\bm{y}',\bm{a}\right) &\defeq \left\{(\bm{z},\bm{y})\in \mf_t(\R^d\times[1:q]): z(s) \in \Dmc\left(y(s),y'[s-1:s],a[s-1:s]\right)\right\}.\label{eq::Dbar}
\end{align}

Note that for any $s\in[1:t]$ and $(\bm{y}',\bm{a}) \in \mf_t([1:q])\times \mf_t([m])$,
\begin{align*}
\PP\left((Z\pcar{1}[t],Y\pcar{1}[t]) \in \partial \bar{\Dmc}_{s,t}\left(\bm{y}',\bm{a}\right)\right) &= \PP\left(Z\pcar{1}(s) \in \partial\Dmc\left(Y\pcar{1}(s),y'[s-1:s],a[s-1:s]\right)\right)\\
&= \sum_{y =1}^q\PP\left(Z\pcar{1}(s) \in \partial \Dmc\left(y,y'[s-1:s],a[s-1:s]\right),Y\pcar{1}(s)=y\right)\\
&\leq \sum_{y=1}^q \PP\left(Z\pcar{1}(s) \in \partial \Dmc\left(y,y'[s-1:s],a[s-1:s]\right)\right)\\
&= 0,
\end{align*}
where the last equality holds by Assumption \ref{assu::Bbds} and because $Z\pcar{1}(s)$ is absolutely continuous by Lemma \ref{lem::abscont}. By a near identical argument, $\partial \bar{\Dmc}_{0,t}(\bm{y}',\bm{a})$ is also a $(Z\pcar{1}[t],Y\pcar{1}[t])$-null set.

\ind Now, consider the finite $\sigma$-algebra $\Amc_t \defeq \sigma\left(\bar{\Dmc}_{s,t}(\bm{y}',\bm{a}):s \in [t], (\bm{y}',\bm{a}) \in \mf_t([1:q])\times\mf_t([m])\right)$. For any set $A \in \Amc_t$, 
\begin{align*}
\PP\left((Z\pcar{1}[t],Y\pcar{1}[t]) \in \partial A\right) &\leq \PP\left((Z\pcar{1}[t],Y\pcar{1}[t]) \in \bigcup_{s=0}^t\bigcup_{(\bm{y}',\bm{a}) \in \mf_t([1:q]\times[m])} \partial \bar{\Dmc}_{s,t}(\bm{y}',\bm{a}) \right)\\
&\leq \sum_{s=0}^t \sum_{(\bm{y}',\bm{a}) \in \mf_t([1:q]\times[m])} \PP\left((Z\pcar{1}[t],Y\pcar{1}[t]) \in \partial \bar{\Dmc}_{s,t}(\bm{y}',\bm{a})\right)\\
&= 0.
\end{align*}
Therefore, it suffices to prove that there exists a set $\Emc\in \Amc_t$ satisfying \eqref{eqn::AGE} and \eqref{eqn::AGEc}.

\ind To find $\Emc$, recall from Lemma \ref{lem::support} that the support of $V \defeq (Z\pcar{1,2},Y\pcar{1,2},A)$ is given by the set $S_t$ defined in \eqref{eqn:S_t}; for any $v \defeq (\bm{z}_{1,2},\bm{y}_{1,2},\bm{a}) \in \mf_t(\R^d)^2\times\mf_t([1:q])^2\times \mf_t([m])$, $v \notin S_t$ if and only if at least one of the following conditions are satisfied:
\begin{align*}
P_0(z_1(0),y_1(0))\prod_{s=1}^t P(z_1(s),y_1[s-1:s]) &= 0,\\
B_0(z_{1,2}(0),y_{1,2}(0),a(0))P_0(z_2(0),y_2(0)) &= 0,\\
B(z_{1,2}(s),y_{1,2}(s),a[s-1:s])P(z_2(s),y_2[s-1:s]) &= 0\te{ for some }s \in [1:t].
\end{align*}
Let $T[t] \defeq (Z\pcar{1}[t],Y\pcar{1}[t],A\pcar{12}[t])$, and let $\bar{S}_t \defeq \te{Supp}(T[t])$. We can apply the above display to describe $\bar{S}_t$ in terms of the sets $\bar{D}_{s,t}(\cdot)$ defined in \eqref{eq::Dbar}. For any $\bm{t} \defeq (\bm{z},\bm{y},\bm{a}) \in \mf_t(\R^d\times[1:q]\times [m])$, $\bm{t} \notin \bar{S}_t$ if for every $\bm{y}'\in \mf_t([1:q])$, at least one of the following conditions are satisfied:
\begin{equation}
\label{eqn::imposs}
P_0(z(0),y(0))\prod_{s'=1}^t P(z(s'),y[s'-1,s']) =0 \te{ or } (\bm{z},\bm{y}) \in \bigcup_{s=0}^t\bar{D}_{s,t}(\bm{y}',\bm{a}).
\end{equation}
Let $O$ be the set of $(\bm{z},\bm{y})$ satisfying the first part of \eqref{eqn::imposs}. Then,
\begin{equation}
\label{eqn::Sbartdef}
\bar{S}_t^c = \left(O\times \mf_t([m])\right)\cup \bigcup_{\bm{a} \in \mf_t([m])} \bigcap_{\bm{y}' \in \mf_t([1:q])}\bigcup_{s=0}^t \bar{D}_{s,t}(\bm{y}',\bm{a})\times \{\bm{a}\}.
\end{equation}
As a consequence of Lemma \ref{lem::support}, the density of the random element $T[t]$ with respect to the product of the Lebesgue measure on $\mf_t(\R^d)$ and the uniform measure on $[1:q]\times[m]$ has a strictly positive density on $\bar{S}_t$.

\ind Consider the set
\begin{equation}
\label{eqn::Emcdef}
\Emc \defeq \bigcap_{\substack{\bm{a} \in \mf_t([m])\\a(t) = k}}  \bigcap_{\bm{y}'\in \mf_t([1:q])} \bigcup_{s=0}^t \bar{\Dmc}_{s,t}(\bm{y}',\bm{a}).
\end{equation} 
Then for any $\bm{t}$ such that $(\bm{z},\bm{y}) \in \Emc$ and $\bm{a}(t) = k$, $\bm{t} \in \bar{S}^c_t$. Moreover, for any $(\bm{z},\bm{y}) \notin \Emc$, there must exist an $\bm{a} \in \mf_t([m])$ such that $\bm{a}(t) = k$ and for some $\bm{y}' \in \mf_t([1:q])$, $(\bm{z},\bm{y}) \notin \bar{D}_{s,t}(\bm{y}',\bm{a})$ for all $s \in [t]$. This implies that $\bm{t} \defeq (\bm{z},\bm{y},\bm{a}) \in \bar{S}_t\cup O\times \mf_t([m])$. 

\ind We can now verify that $\Emc$ is indeed the set we are seeking. By \eqref{eqn::Emcdef}, $\Emc \in \Amc_t$, so $\partial\Emc$ is a null set as desired. Furthermore, as shown above, given any $\bm{t} \defeq (\bm{z},\bm{y},\bm{a})$ such that $(\bm{z},\bm{y})\in \Emc$ and $\bm{a}(t) = k$, $\bm{t} \in \bar{S}^c_t$. So,
\begin{align*}
\ex{\PP\left(A\pcar{12}(t) = k\middle|\Gmc^1_t\right)\indic{(Z\pcar{1}[t],Y\pcar{1}[t]) \in \Emc}} = \PP\left(A\pcar{12}(t) = k,(Z\pcar{1}[t],Y\pcar{1}[t])\in \Emc\right) = 0,
\end{align*}
so $\Emc$ satisfies \eqref{eqn::AGE}. Next, for any $\bm{t}$ such that $(\bm{z},\bm{y})\notin \Emc$ and $\bm{a}(t) = k$, $\bm{t} \in \bar{S}_t\cup O\times \mf_t([m])$. However, since the density of $T[t]$ is strictly positive on $\bar{S}_t$, it follows that $\PP(A\pcar{12}(t) = k|\Gmc_t^1)> 0$ whenever there exists an $\bm{a}$ such that $\bm{a}(t) = k$ and $(Z\pcar{1}[t],Y\pcar{1}[t],\bm{a}) \in \bar{S}_t$. Therefore, $\PP(A\pcar{12}(t) = k|\Gmc_t^1)> 0$ only when $(Z\pcar{1}[t],Y\pcar{1}[t]) \in \Emc\cup O$. So,
\begin{align*}
\PP\left(\PP\left(A\pcar{12}(t) = k\middle|\Gmc^1_t\right) = 0,(Z\pcar{1}[t],Y\pcar{1}[t]) \notin \Emc\right) &\leq \PP\left((Z\pcar{1}[t],Y\pcar{1}[t]) \in O\right) = 0.
\end{align*}
This concludes the proof of the lemma.
\end{proof}

\appendix

\section{Sufficient Conditions for Uniform Integrability}
\label{sec::ui}
In this section, we establish sufficient conditions for Assumption \ref{assu::ui} to hold. Our sufficient conditions are separated into two cases: static/mixing and general. Either way, we start with the assumption that the latent positions and noise terms at time $0$ have finite exponential moments:

\begin{assumption}[Finite Exponential Moments]
\label{assu:expbds}
The following assumptions hold:
\begin{enumerate}[label = (\alph*)]
\item for each $n \in \N$,
\begin{equation}
\label{eq::MGFZ0}
\limsup_{n\to\infty} \ex{\exp\left(C|Z^n_1(0)|\right)} < \infty \te{ for all }C \in \R;
\end{equation}
\item likewise,
\begin{equation}
\label{eq::MGFxi0}
 \ex{\exp\left(C|\xi_1(0)|\right)} < \infty \te{ for all }C \in \R;
\end{equation}
\end{enumerate}
\end{assumption}

\ind In general, our assumptions ensure that the moment generating functions (MGF) of the latent opinions at any given time $t\geq 0$ have a finite limsup, which immediately implies uniform integrability of $\{Z^n_i(t)\}_{n \in\N, i \in [1:n]}$. To do this, we use an inductive argument; for each $t$, if $Z^n_i(t)$ has a finite MGF in the limit, then so does $L^{n,k}_i(t)$ for each $k \in [1:m]$. If each $L^{n,k}_i(t)$ has a finite MGF in the limit, then so does $Z^n_i(t+1)$. The second part of the argument is a rather straightforward calculation, but the first is more complicated. We present two mutually exclusive types of sufficient conditions expressed as bounds on the interaction kernels $B_0$ and $B$. For each $k, y$, and $y'$, all interactions of type $k$ between particles of respective type $y$ and $y'$ must satisfy one of these conditions.

\ind Our first type of sufficient condition is that the respective networks with adjacency matrices $A^{n,k}_{ij}(t)\defeq \indic{A^n_{ij}=k}(t)$ are sufficiently well connected; specifically that the probability of two particles interacting decreases at most exponentially fast in the distance between the latent positions of the particles. In this case, it can be shown via a careful computation that $L^{n,k}_i(t)$, which represents the average latent opinion of the neighbors of $i$ at time $t$, will have a small enough spread to possess a finite moment generating function (in the limit). 

\ind The second type of sufficient condition is that two particles of a given type pair can only experience type $k$ interactions if the distance between their latent positions is bounded. If this condition holds for all interactions of a given type $k$, then $|L^{n,k}_i(t) - Z^n_i(t)|$ is a.s. bounded, so the MGF of $L^{n,k}_i(t)$ is bounded by a constant multiple of the (finite) MGF of $Z^n_i(t)$. It is worth nothing that this condition allows for restrictions preventing interactions of a given type $k$ between particle pairs with respective types $y$ and $y'$.

\ind In general, for any $k$, we require all type-$k$ interactions to satisfy one type of sufficient condition regardless of particle types. However, if we place a few additional assumptions on $P$ and $P_0$, then we may allow each interaction of type $k$ to satisfy one or both of the above sufficient conditions depending on the types of the particle pair. Essentially, we must assume either that particle types are static or particle types quickly mix.

\ind In Section \ref{ssec::uisc}, we state the sufficient conditions in both cases. Then in Section \ref{ssec::uipf} we prove that these conditions are sufficient for Assumption \ref{assu::ui} to hold.

\subsection{Some sufficient conditions}
\label{ssec::uisc}

As mentioned above, we present two different types of bounds.

\begin{definition}
\label{defn::Bbds}
For any $(y_{\{1,2\}},k) \in [1:q]^2\times[m]$, 
\begin{enumerate}[label = (\alph*)]
\item we say the model satisfies a $(y_{\{1,2\}},k)$\emph{-exponential bound with respect to } $C_k > 0$ if 
\begin{align}
\inf_{z_{\{1,2\}}\in(\R^d)^2} B_0(&z_{\{1,2\}},y_{\{1,2\}},k)\exp\left(C_k|z_1-z_2|\right) > 0,\label{eq::B0expbd}\\
\inf_{(z_{\{1,2\}},k_0)\in(\R^d)^2\times [m]} B(&z_{\{1,2\}},y_{\{1,2\}},k_0,k) \exp\left(C_k|z_1-z_2|\right) > 0.\label{eq::Bexpbd}
\end{align} 

\item Likewise, we say the model satisfies a $(y_{\{1,2\}},k)$\emph{-finite range bound with respect to } $C_k > 0$ if for all $k_0\in [m]$ and $z_{\{1,2\}} \in (\R^d)^2$ satisfying $|z_1 - z_2|> C_k$,
\begin{equation}
\label{eq::B0Bcutoff}
B_0(z_{\{1,2\}},y_{\{1,2\}},k) = B(z_{\{1,2\}},y_{\{1,2\}},k_0,k) = 0.
\end{equation}
\end{enumerate}
\end{definition}

With this, we may state the general assumption.

\begin{assumption}[Uniform Integrability Bounds (General Case)]
\label{assu::bdsgen}
For each $k \in [1:m]$, there exists a constant $C_k > 0$ such that either of the following statements hold for all $y_{\{1,2\}} \in [1:q]^2$: 
\begin{enumerate}[label = (\alph*)]
\item the model satisfies a type $(y_{\{1,2\}},k)$-exponential bound with respect to $C_k$, or
\item the model satisfies a type $(y_{\{1,2\}},k)$-finite range bound with respect to $C_k$.
\end{enumerate}
\end{assumption}

We may weaken Assumption \ref{assu::bdsgen} by placing the following additional assumption on the type evolutions:

\begin{assumption}
\label{assu::ybds}
Either
\begin{enumerate}[label = (\roman*)]
\item the particle types are static and non-vanishing in the following sense: there exists a $\delta > 0$ such that
\[P_0(z,y) \geq \delta \te{ and } P(z,y,y') \equiv 0 \te{ for all }(z,y',y) \in \R^d\times[1:q]^2\te{ s.t. }y\neq y',\]
\item or particle types are dynamic and quickly mixing in the following sense: there exists a $\delta > 0$ such that
\[P_0(z,y) \geq \delta \te{ and } P(z,y',y) \geq \delta \te{ for all } (z,y',y) \in \R^d\times[1:q]^2.\]
\end{enumerate}
\end{assumption}

Given Assumption \ref{assu::ybds}, we may replace Assumption \ref{assu::bdsgen} by,

\begin{assumption}[Uniform Integrability Bounds (Static or Mixing Case)]
\label{assu::bdsy}
For each $(y_{\{1,2\}},k) \in [1:q]^2\times[m]$, there exists a constant $C_k > 0$ such that either of the following statements hold: 
\begin{enumerate}[label = (\alph*)]
\item the model satisfies a type $(y_{\{1,2\}},k)$-exponential bound with respect to $C_k$, or
\item the model satisfies a type $(y_{\{1,2\}},k)$-finite range bound with respect to $C_k$.
\end{enumerate}
\end{assumption}

With this, we have the following result.

\begin{proposition}
\label{prop:uisuffconds}
Suppose Assumptions \ref{assu::cond}, \ref{assu::init}, \ref{assu::Bbds} and \ref{assu:expbds} hold. Assume also that either Assumption \ref{assu::bdsgen} or Assumptions \ref{assu::ybds} and \ref{assu::bdsy} hold. Then Assumption \ref{assu::ui} holds.
\end{proposition}

\subsection{Proof of Proposition \ref{prop:uisuffconds}}
\label{ssec::uipf}

As mentioned above in Section \ref{sec::ui}, we actually prove the following slightly stronger statement.

\begin{lemma}
\label{lem::unifind}
Under the conditions of Proposition \ref{prop:uisuffconds}, 
\begin{equation}
\label{eqn::zsupexpLn}
\limsup_{n\to\infty} \ex{\exp\left(C|Z^n_1(t)|\right)} < \infty.
\end{equation}
\end{lemma}

Proposition \ref{prop:uisuffconds} immediately follows from the exponential moment bound and exchangeability of $\{Z^n_i(t)\}_{i\in [1:n]}$.
\begin{proof}

\ind Our primary approach is to apply a joint inductive approach to establish \eqref{eqn::zsupexpLn} and the following equation:
\begin{equation}
\label{eqn::lsupexpLn}
\limsup_{n\to\infty} \ex{\exp\left(C|L^{n,k}_1(t)|\right)} < \infty \te{ for all } k \in [1:m], C > 0.
\end{equation}
The $t = 0$ base case of \eqref{eqn::zsupexpLn} follows from Assumption \ref{assu::ui}(a). For our first inductive argument, we show that \eqref{eqn::zsupexpLn} at time $t$ implies \eqref{eqn::lsupexpLn} at time $t$. In our second inductive argument, we then show that \eqref{eqn::zsupexpLn} and \eqref{eqn::lsupexpLn} at time $t$ together imply \eqref{eqn::zsupexpLn} at time $t+1$. These two combined imply the result. The proof of the first inductive argument is further subdivided as follows.

\ind Fix $k \in [1:m]$ and define $C_k > 0$ so that all bounds in Assumptions \ref{assu::bdsgen}-\ref{assu::bdsy} hold uniformly over $C_k$. 

\ind Let $\Emc$ denote the set of elements $y \in [1:q]$ for which there exists at least one element $y' \in [1:q]$ such that the model satisfies a $(y,y',k)$-exponential bound with respect to $C_k$. For each $y \in \Emc$, let $\Emc_y$ be the set of $y'\in [1:q]$ such that the model satisfies a $(y,y',k)$-exponential bound with respect to $C_k$. Finally, consider the event
\[\Emc^n_Z \defeq \{Y^n_1(t) \in \Emc\}\cap \left\{\te{there exists }j > 2\te{ s.t. } Y^n_j(t) \in \Emc_{Y^n_1(t)}\right\}.\]

$\Emc^n_Z$ tells us that there exists at least one particle in $[3:n]$ for which the $\Fmc^{A,n}_t$-conditional probability of a type $k$ interaction between that particle and particle 1 satisfies the exponential bound of Definition \ref{defn::Bbds}.

\begin{remark}
\label{rem:Emc}
Under Assumption \ref{assu::bdsgen}, $\Emc = [1:q]$ or $\Emc = \emptyset$. If $\Emc = \emptyset$, then $\Emc_y = \emptyset$ for every $y \in [1:q]$ and $\Emc^n_Z = \emptyset$. If $\Emc = [1:q]$, then $\Emc_y = [1:q]$ for all $y \in [1:q]$ and $\Emc^n_Z = \Omega$. Therefore, these sets and events only play a non-trivial role under Assumptions \ref{assu::ybds} and \ref{assu::bdsy}.
\end{remark}

\textbf{Proof of the first inductive argument:}
Define the function $\psi_C: \R^d \to \R$ by $\psi_C(z) = \exp(C|z|)$. It is easily verified that $\psi_C$ is convex in $z$ and increasing and superlinear in $|z|$ for all $C > 0$. Define $M:\R_+\to\R_+$ by
\begin{equation}
\label{eq:Mdef}
M(C) \defeq \limsup_{n\to\infty} \ex{\psi_C(Z^n_1(t))}.
\end{equation}

Then we may write

\[\limsup_{n\to\infty}\ex{\psi_C(L^{n,k}_1(t))} = \limsup_{n\to\infty}\ex{\psi_C(L^{n,k}_1(t))\indic{\Emc^n_Z}} + \limsup_{n\to\infty}\ex{\psi_C(L^{n,k}_1(t))\indic{(\Emc_Z^n)^c}},\]
so we may bound the two limits above separately.\\

\textbf{First Inductive Argument: The $\Emc^n_Z$ Case:}

\ind Fix $C > 0$. Then applying the convexity of $\psi_C$ and the fact that $\sum_{j=1}^n \frac{A^{n,k}_{1j}(t)}{\sum_{i=1}^n A^{n,k}_{1i}(t)} = 1$, we can break up the expectation of $\ex{\psi_C(L^{n,k}_1(t))}$ as follows:
\begin{equation}
\label{eqn::breakupsum}
\ex{\psi_C(L^{n,k}_1(t))} = \ex{\psi_C\left(\frac{1}{\sum_{i=1}^n A^{n,k}_{1i}(t)}\sum_{j=1}^n A^{n,k}_{1j}(t)Z^n_j(t)\right)} \leq \sum_{j=1}^n\ex{\frac{A^{n,k}_{1j}(t)}{\sum_{i=1}^n A^{n,k}_{1i}(t)}\psi_C\left(Z^n_j(t)\right)}.
\end{equation}
The $j=1$ term of the sum above is easily reduced using the fact that $A^{n,k}_{11}(t)=0$ (and using the $0/0 = 0$ convention), so:
\begin{equation}
\label{eqn::j=1breakup}
\limsup_{n\to\infty}\ex{\frac{A^{n,k}_{11}(t)}{\sum_{i=1}^n A^{n,k}_{1i}(t)}\psi_C\left(Z^n_1(t)\right)}=0.
\end{equation}
We can now examine the remaining terms. Recall the $\sigma$-algebras $\Fmc^{A,n}_t \defeq \sigma(Z^n[t],Y^n[t],A^n[t-1])$ if $t > 0$, where $\Fmc^{A,n}_0 \defeq \sigma(Z^n(0),Y^n(0))$. For $j \in [1:n]$, \eqref{eqn::Anevolve} and Assumption \ref{assu::init}(b) imply
\[P^n_j \defeq \ex{A^{n,k}_{1j}(t)|\Fmc^{A,n}_t} = \begin{cases}
B\left(Z^n_{\{1,j\}}(t),Y^n_{\{1,j\}}(t),A^n_{1j}(t-1),k\right)  &\te{ if } t > 0\te{ and }j\neq 1,\\
B_0(Z^n_{\{1,j\}}(0),Y^n_{\{1,j\}}(0))&\te{ if } t = 0\te{ and }j\neq 1,\\
0 &\te{ if } j = 1.
\end{cases}\]
Define
\[P^n_{-2} =\frac{1}{n-2}\sum_{i> 2} P^n_i.\]
Note that $\Emc^n_Z$ is $\Fmc^{A,n}_t$-measurable and $\Emc_{Y^n_1(t)}$ is a $\Fmc^{A,n}_t$-measurable random subset of $[1:q]$. Moreover, $\Emc^n_Z \subseteq \{P^n_{-2} > 0\}$. Therefore, if we multiply the expression inside the expectation by $\indic{\Emc^n_Z}$, we may repeat nearly the exact computation from (7.6)-(7.12) of \cite{GanSpiSus25} to get
\begin{equation}
\label{eq:AZtoPZ}
\limsup_{n\to\infty} \sum_{j=2}^n \ex{\frac{A^{n,k}_{1j}(t)}{\sum_{i=1}^n A^{n,k}_{1i}(t)}\psi_C(Z^n_j(t))\indic{\Emc^n_Z}} \leq \sqrt{M(2C)}\limsup_{n\to\infty}\sqrt{\ex{\frac{\indic{\Emc^n_Z}}{(P^n_{-2})^2}}}.
\end{equation}

By definition of $\Emc^n_Z$ and Definition \ref{defn::Bbds}, there exists a constant $\bar{C}> 0$ such that for all $n \in \N$ and $j \in [1:n]$,
\[P^n_j \geq \bar{C}\exp\left(-C_k|Z^n_1(s) - Z^n_j(s)|\right)\indic{Y^n_j(t)\in \Emc_{Y^n_1(t)}},\]
and on the event $\Emc^n_Z$, $P^n_j > 0$ for at least one value of $j$. It is therefore useful to keep track of the values of $j$ for which $Y^n_j(t) \in \Emc_{Y^n_1(t)}$:
\[U_n \defeq \{j \in [3:n]: Y^n_j(t) \in \Emc_{Y^n_1(t)}\},\]
By definition, $U_n$ is a random subset of $[3:n]$, and on the event $\Emc^n_Z$, $U_n$ is nonempty. Therefore,
\[P^n_{-2}\indic{\Emc^n_Z} \geq \frac{\bar{C}\indic{\Emc^n_Z}}{n-2}\sum_{j \in U_n} \exp\left(-C_k|Z^n_1(s)-Z^n_j(s)|\right).\]
We can now perform the following computation, where we apply the AM-GM inequality once more in \eqref{eqn::PsqinvAMGM}.
\begin{align}
\frac{\indic{\Emc^n_Z}}{(P^n_{-2})^2} &\leq \frac{\indic{\Emc^c_Z}}{\left(\frac{\bar{C}}{n-2}\sum_{j\in U_n} \exp\left(-C_k|Z^n_1(t) - Z^n_j(t)|\right)\right)^2}\nonumber\\
&=\frac{\indic{\Emc^n_Z}}{\frac{\bar{C}^2}{(n-2)^2}\sum_{i,j\in U_n} \exp\left(-C_k\left(|Z^n_1(t) - Z^n_i(t)| + |Z^n_1(t) - Z^n_j(t)|\right)\right)}\nonumber\\
&=\left(\frac{1}{\bar{C}^2}\right)\frac{\indic{\Emc^n_Z}}{\frac{1}{(n-2)^2}\sum_{i,j\in U_n} \exp\left(-C_k\left(|Z^n_1(t) - Z^n_i(t)| + |Z^n_1(t) - Z^n_j(t)|\right)\right)}\nonumber\\
&=\left(\frac{1}{\bar{C}^2}\right)\left(\frac{n-2}{|U_n|}\right)^2\frac{\indic{\Emc^n_Z}}{\frac{1}{|U_n|^2}\sum_{i,j\in U_n} \exp\left(-C_k\left(|Z^n_1(t) - Z^n_i(t)| + |Z^n_1(t) - Z^n_j(t)|\right)\right)}\nonumber\\
&\leq\left(\frac{1}{\bar{C}^2}\right)\left(\frac{n-2}{|U_n|}\right)^2\frac{\indic{\Emc^n_Z}}{\exp\left(-\frac{C_k}{|U_n|^2}\sum_{i,j\in U_n}\left(|Z^n_1(t) - Z^n_i(t)| + |Z^n_1(t) - Z^n_j(t)|\right)\right)}\label{eqn::PsqinvAMGM}\\
&=\left(\frac{\indic{\Emc^n_Z}}{\bar{C}^2}\right)\left(\frac{n-2}{|U_n|}\right)^2\exp\left(\frac{C_k}{|U_n|^2}\sum_{i,j\in U_n}\left(|Z^n_1(t) - Z^n_i(t)| + |Z^n_1(t) - Z^n_j(t)|\right)\right)\nonumber\\
&=\left(\frac{\indic{\Emc^n_Z}}{\bar{C}^2}\right)\left(\frac{n-2}{|U_n|}\right)^2\exp\left(\frac{C_k(2|U_n|-1)}{|U_n|^2}\sum_{j\in U_n}|Z^n_1(t) - Z^n_j(t)|\indic{j\in U_n}\right)\\
&\leq\left(\frac{\indic{\Emc^n_Z}}{\bar{C}^2}\right)\left(\frac{n-2}{|U_n|}\right)^2\exp\left(\frac{2C_k}{|U_n|}\sum_{j\in U_n}|Z^n_1(t) - Z^n_j(t)|\right).\label{eqn::Psqinv}
\end{align}

To bound the above term in expectation, we introduce the following lemmas which we prove in the next section.

\begin{lemma}[Random H\"{o}lder's Inequality]
\label{lem::Holder}
Let $I$ be a finite index set. Let $S$ be a random subset of $I$, and let $\{X_i\}_{i \in I}$ be a sequence of random variables such that the sequence $\{(X_i,\indic{i\in S})\}_{i\in I}$ is exchangeable. Then for any $j \in I$,
\[\ex{\prod_{i \in S} X_i^{1/|S|}}\leq \ex{|S|}\ex{\frac{X_j}{|S|}\middle|j \in S}+1.\]
\end{lemma}

\begin{lemma}
\label{lem::invbd}
For any $\ell \in \N$,
\[\limsup_{n\to\infty}\ex{\left(\frac{n-2}{|U_n|}\right)^{\ell}\indic{|U_n|>0}} \defeq \limsup_{n\to\infty} \Mmc_{n,\ell} < \infty.\]
Moreover, for $n$ sufficiently large, $\ex{|U_n|} \geq (n-2)\delta^2$.
\end{lemma}

We also apply the following simple argument. By exchangeability of $\{\indic{i \in U_n}\}_{i=3}^n$ and by Lemma \ref{lem::invbd}, $\PP(3 \in U_n) = \ex{\PP(3 \in U_n| |U_n|)} = \frac{\ex{|U_n|}}{n-2} \geq \delta/2$ for $n$ sufficiently large. So for any non-negative random variable $X$,
\begin{equation}
\label{eq::condbd}
\ex{X|3 \in U_n}\leq \frac{\ex{X}}{\PP(3\in U_n)} \leq \frac{\ex{X}}{\delta^2}.
\end{equation}

Assume that $n$ is sufficiently large for the statements of Lemma \ref{lem::invbd} to hold. Below, in \eqref{eq::longholder}, we apply H\"older's inequality. In \eqref{eq::mmcx}, we apply Lemma \ref{lem::invbd} and we prepare to apply Lemma \ref{lem::Holder} with $I = [3:n]$, $S = U_n$ and $X_i = \exp(4C_k|Z^n_1(t)-Z^n_i(t)|)$. The exchangeability of $(X_i,\indic{i \in S})$ over $i \in I$ follows from the exchangeability excluding 1 of $(Z_i,Y_i,A_{1i})$, $i = 2,\dots,n$ and the fact that $|U_n|$ and $\Emc^n_Z$ are invariant with respect to permutations of vertices in $[3:n]$. In \eqref{eq::2delt}, we apply \eqref{eq::condbd} twice and use the fact that $|U_n|\leq n-2$. In \eqref{eq::mmcn2}, we again apply Lemma \ref{lem::invbd}. Finally, in \eqref{eq::longfinish}, we apply the argument from (7,16)-(7.19) of \cite{GanSpiSus25}.

\begin{align}
\nonumber
\ex{\frac{\indic{\Emc^n_Z}}{(P^n_{-2})^2}} &\leq \ex{\left(\frac{\indic{\Emc^n_Z}}{\bar{C}^2}\right)\left(\frac{n-2}{|U_n|}\right)^2\exp\left(\frac{2C_k}{|U_n|}\sum_{j\in U_n}|Z^n_1(t) - Z^n_j(t)|\indic{j\in U_n}\right)}\\
\label{eq::longholder}
&\leq \ex{\left(\frac{n-2}{|U_n|\bar{C}}\right)^4\indic{\Emc^n_Z}}\ex{\exp\left(\frac{4C_k}{|U_n|}\sum_{j \in U_n} |Z^n_1(t) - Z^n_j(t)|\right)}\\
\label{eq::mmcx}
&\leq \Mmc_{n,4}\ex{\prod_{i \in S}X_i^{1/|S|}}\\\nonumber
&\leq \Mmc_{n,4}\ex{|S|}\ex{\frac{X_3}{|S|}\middle|3 \in S} + \Mmc_{n,4}\\\nonumber
&\leq \Mmc_{n,4}\ex{|U_n|}\ex{\frac{1}{|U_n|}\exp(4C_k|Z^n_1(t) - Z^n_3(t)|)\middle|3 \in U_n}+\Mmc_{n,4}\\\nonumber
&\leq \Mmc_{n,4}\ex{|U_n|}\ex{\frac{1}{|U_n|^2}\middle|3\in U_n}\ex{\exp\left(8C_k|Z^n_1(t)-Z^n_3(t)|\right)\middle|3 \in U_n}+\Mmc_{n,4}\\
\label{eq::2delt}
&\leq \Mmc_{n,4}\ex{\frac{n-2}{|U_n|^2}}(1/\delta^2)^2\ex{\exp(8C_k|Z^n_1(t) - Z^n_3(t)|)}+\Mmc_{n,4}\\
\label{eq::mmcn2}
&\leq \Mmc_{n,4}\Mmc_{n,2}(1/\delta^4)\ex{\exp(8C_k|Z^n_1(t) - Z^n_3(t)|)}+\Mmc_{n,4}\\
\label{eq::longfinish}
&\leq \frac{\Mmc_{n,4}\Mmc_{n,2}M(16C_k)}{\delta^4}+\Mmc_{n,4}.
\end{align}

Combining \eqref{eqn::breakupsum}, \eqref{eqn::j=1breakup}, \eqref{eq:AZtoPZ} and \eqref{eq::longfinish} yields

\begin{align*}
\limsup_{n\to\infty} \ex{\psi_C(L^{n,k}_1(t))\indic{\Emc^n_Z}} &\leq \sum_{j=1}^n \ex{\frac{A^{n,k}_{1j}(t)}{\sum_{i=1}^n A^{n,k}_{1i}(t)}\psi_C(Z^n_j(t))\indic{\Emc^n_Z}}\\
&= \sqrt{M(2C)}\limsup_{n\to\infty} \sqrt{\ex{\frac{\indic{\Emc^n_Z}}{(P^n_{-2})^2}}}\\
&= \sqrt{M(2C)}\limsup_{n\to\infty} \left(\frac{\Mmc_{n,4}\Mmc_{n,2}M(16C_k)}{\delta^4}+\Mmc_{n,4}\right)\\
&< \infty.
\end{align*}

\textbf{First Inductive Argument: The $(\Emc^n_Z)^c$ Case:} 

Consider the event $(\Emc^n_Z)^c$. When this event is true, it follows that for each $i$, either the $\Fmc^{A,n}_t$-conditional probability of a type $k$ interaction between particles $1$ and $i$ satisfies a finite range bound in the sense of Definition \ref{defn::Bbds}, or $1$ and $i$ do not have a type $k$ interaction. Therefore, if $A^n_{1i}(t) = k$, then $|Z^n_j(t)|\leq |Z^n_1(t)| + C_k$. Applying the convention that $0/0 = 0$, it follows that
\begin{align*}
\left|L^{n,k}_1(t)\right| &\leq \frac{1}{\sum_{j=1}^n A^{n,k}_{1j}(t)}\sum_{j=1}^n \left|Z^n_j(t)\right|A^{n,k}_{1j}(t) \\
&\leq\frac{1}{\sum_{j=1}^n A^{n,k}_{1j}(t)}\sum_{j=1}^n\left( \left|Z^n_1(t)\right|+C_k\right)A^{n,k}_{1j}(t)\\
&= \left|Z^n_1(t)\right| + C_k. 
\end{align*}
Thus, for any $C > 0$,
\[\limsup_{n\to\infty} \ex{\psi_C\left(L^{n,k}_1(t)\right)\indic{(\Emc^n_Z)^c}}\leq \left(\limsup_{n\to\infty} \ex{\psi_C\left(Z^n_1(t)\right)}\right)\exp(CC_k) = M(C)\exp(CC_k) < \infty.\]\\

\textbf{Second Inductive Argument:} We show that \eqref{eqn::zsupexpLn} and \eqref{eqn::lsupexpLn} at time $s$ imply \eqref{eqn::zsupexpLn} at time $s+1$. Note that
\[Z^n_1(t+1) - \xi^1(t) = \Lambda_0(Y^n_1(t))Z^n_1(t) + \sum_{k=1}^m \Lambda_k(Y^n_1(t)) L^{n,k}_1(t).\]
For each $k \in [m]$, let $N_{n,k}=\|\Lambda_k(Y^n_1(t))\|_{\te{op}}$. Then for any vector $v \in \R^d$, $|\Lambda_kv|\leq N_{n,k}|v|< \infty$. In addition, if we define $N_n \defeq \sum_{k=0}^m N_{n,k}$, then for each $n$, either $N_n > 0$ or $Z^n_1(t+1) = \xi^1(t)$ (in which case the result is trivial). $N_{n,k}$, $n \in \N,k\in[m]$, is a sequence of finite random variables with a fixed support of size at most $q$ (depending on $k$ but not depending on $n$), so there exists a constant $\bar{N}<\infty$ such that $N_{n,k} \leq \bar{N}$ a.s. for all $n \in \N$ and $k \in [m]$. If $N_n = 0$ a.s., then $Z^n_1(s+1) = \xi^1(t)$, so uniform integrability follows trivially from Assumption \ref{assu::ui}(b). Therefore, we assume that $N_n > 0$ with positive probability and let $\underline{N}$ be the minimal positive value in the support of all $N_n$.

\ind Utilizing the convexity of the exponential function, the independence of $(Z^n(t+1)-\xi^1(t))$ and $\xi^1(t)$, and Assumption \ref{assu::ui}(b), we make the following computation:
\begin{align*}
&\limsup_{n\to\infty}\ex{\psi_C(Z^n_1(s+1))} \\
& \leq \ex{\psi_C(\xi^1(t))}\limsup_{n\to\infty}\ex{\psi_C(Z^n_1(t+1)-\xi^1(t))}\\
&= \ex{\psi_C(\xi^1(t))}\limsup_{n\to\infty}\ex{\psi_C\left(\Lambda_0(Y^n_1(t))Z^n_1(t) + \sum_{k=1}^m \Lambda_k(Y^n_1(t))L^{n,k}_1(t)\right)}\\
&\leq \ex{\psi_C(\xi^1(t))}\limsup_{n\to\infty}\ex{\exp\left(CN_n\left(\frac{N_{n,0}}{N_n}|Z^n_1(t)| + \sum_{k=1}^m \frac{N_{n,k}}{N_n}|L^{n,k}_1(t)|\right)\indic{N_n>0} +\indic{N_n=0}\right)}\\
&\leq \ex{\psi_C(\xi^1(t))}\\
&\ind\ind\limsup_{n\to\infty}\Emb\Bigg[\frac{N_{n,0}}{N_n}\exp\left(CN_n\left(|Z^n_1(t)|\right)\right)\indic{N_n>0} + \sum_{k=1}^m \frac{N_{n,k}}{N_n}\exp\left(CN_n\left(|L^{n,k}_1(t)|\right)\right)\indic{N_n>0} + e\indic{N_n=0}\Bigg]\\
&\leq \ex{\psi_C(\xi^1(t))}\limsup_{n\to\infty}\ex{\frac{N_{n,0}}{N_{n}}\psi_{CN_{n}}\left(Z^n_1(t)\right)\indic{N_n>0} + \sum_{k=1}^m \frac{N_{n,k}}{N_{n}}\psi_{CN_{n}}\left(L^{n,k}_1(t)\right)\indic{N_n>0} + e\indic{N_n = 0}}\\
&\leq \ex{\psi_C(\xi^1(t))}\left(\frac{\bar{N}}{\underline{N}}M(Cm\bar{N}) + \sum_{k=1}^m \frac{\bar{N}}{\underline{N}}\frac{\sqrt{M(2Cm\bar{N})M(4C_b)}}{C_Z}+e\right)\\
&< \infty,
\end{align*}
where in the fourth and fifth lines of the computation, we abuse notation slightly and apply the convention that $0/0 = 0$ on the event $\{N_n = 0\}$.
\end{proof}

\subsection{Technical Lemmas}
\label{ssec::uitech}

\begin{proof}[Proof of Lemma \ref{lem::Holder}:]
Fix any $j \in I$. Applying the AM-GM inequality and the exchangeability of $(X_j,\indic{j \in S})$, $j \in I$,
\begin{align*}
\ex{\prod_{i\in S} X_i^{1/|S|}} &\leq \ex{1+\frac{\indic{|S|>0}}{|S|}\sum_{i \in S} X_i}\\
&= 1 + \sum_{i\in I}\ex{\frac{\indic{|S|>0}X_i\indic{i \in S}}{|S|}}\\
&= 1 + |I|\ex{\frac{\indic{|S|>0}X_j\indic{j \in S}}{|S|}}\\
&= 1 + |I|\PP(j \in S)\ex{\frac{X_j}{|S|}\middle|j \in S}.
\end{align*}
Applying the tower property and the exchangeability of $\{\indic{j \in S}\}_{j \in I}$,
\[\PP(j \in S) = \ex{\PP\left(j \in S\middle||S|\right)} = \ex{\frac{|S|}{|I|}} = \frac{\ex{|S|}}{|I|}.\]
Combining the above displays, it follows that
\[\ex{\prod_{i\in S} X_i^{1/|S|}} \leq 1 + \ex{|S|}\ex{\frac{X_j}{|S|}\middle|j \in S}\]
as desired.
\end{proof}

\begin{proof}[Proof of Lemma \ref{lem::invbd}:]
Under Assumption \ref{assu::bdsgen}, $|U_n|$ is a random variable with support $\{n-2,0\}$. Either way, $\left(\frac{n-2}{|U_n|}\right)^{\ell}\indic{|U_n|>0} \leq 1$, so $\Mmc_{n,\ell}\leq 1$ and the proof is complete. For the remainder of the proof, we operate under Assumptions \ref{assu::ybds} and \ref{assu::bdsy}.

\ind Let $X^n_i = \indic{i \in U_n}$. Then $|U_n| = \sum_{i=3}^n X_i$. Moreover, Lemma \ref{lem::jexch} implies that the random elements $(Z^n[t],Y^n[t],A^n[t])$ are jointly exchangeable, which implies that $\{X^n_i\}_{i=3}^n$ is exchangeable. To see why, let $\sigma$ be any permutation of the set $[3:n]$. We slightly abuse notation by regarding $\sigma$ as a permutation of $[1:n]$ such that $\sigma(1) = 1$ and $\sigma(2) = 2$. Then,
\[\{X^n_{\sigma(i)}\}_{i=3}^n = \{\indic{Y^n_{\sigma(i)}(t) \in \Emc_{Y^n_1(t)}}\}_{i=3}^n \deq \{\indic{Y^n_{i}(t) \in \Emc_{Y^n_1(t)}}\}_{i=3}^n = \{X^n_i\}_{i=3}^n.\]

\ind As described in the paragraph above \eqref{eqn::peqn}, $\{Y^n_i(t)\}_{i=1}^n$ are assumed to be conditionally independent given $\Fmc^{Y,n}_t$. So if $\PP_n$ and $\Emb_n$ represent the probability and expectation operators conditioned on $\Fmc^{Y,n}_t$, then $\{X^n_i\}_{i=3}^n$ are independent with respect to $\PP_n$. Define 
\[\bar{U}_n \defeq \frac{1}{n-2}\exmu{n}{|U_n|} = \frac{1}{n-2}\sum_{i=3}^n\PP_n\left(X^n_i = 1\right).\]
Note that Assumptions \ref{assu::init}(b), \ref{assu::ybds}, and \eqref{eqn::peqn} all imply the existence of a $\delta > 0$ such that $\PP_n(Y^n_i(t) = y) > \delta$ for all $n \in \N$, $i \in [1:n]$, $t \geq 0$ and $y \in [1:q]$. Therefore, fixing any $y \in \Emc$ and $y' \in \Emc_y$ and applying the conditional independence of $\{Y^n_i(t)\}_{i=3}^n$ given $\Fmc^{Y,n}_t$ described above \eqref{eqn::peqn} (or in Assumption \ref{assu::init}(b) if $t = 0$), 
\[\PP_n(X^n_i = 1) = \PP_n\left(Y^n_1(t) \in \Emc, Y^n_i(t)\in \Emc_{Y^n_1(t)}\right) \geq \PP_n(Y^n_1(t) = y)\PP_n(Y^n_i(t) = y') > \delta^2.\]
So,
\[\bar{U}_n = \exmu{n}{|U_n|} = \frac{1}{n-2}\sum_{i=3}^n \exmu{n}{X^n_i} \geq \delta^2.\]

\ind For $n > \frac{2}{\delta^2}$, the following holds by Hoeffding's inequality:
\begin{equation}
\label{eq::uinvbd1}
\exmu{n}{\left(\frac{n-2}{|U_n|}\right)^{\ell}\indic{1 \leq |U_n|< \frac{(n-2)\bar{U}_n}{2}}} \leq (n-2)^{\ell}\exp\left(-\frac{2\left(\frac{(n-2)\bar{U}_n}{2}\right)^2}{n-2}\right) < (n-2)^{\ell}\exp\left(-\frac{(n-2)\delta^4}{2}\right).
\end{equation}

Likewise, 

\begin{equation}
\label{eq::uinvbd2}
\exmu{n}{\left(\frac{n-2}{|U_n|}\right)^{\ell}\indic{|U_n|\geq \frac{(n-2)\bar{U}_n}{2}}} \leq \left(\frac{2}{\delta^2}\right)^{\ell}.
\end{equation}

Combining equations \eqref{eq::uinvbd1}-\eqref{eq::uinvbd2},

\[\limsup_{n\to\infty} \ex{\left(\frac{n-2}{|U_n|}\right)^{\ell}\indic{|U_n|>0}}\leq \left(\frac{2}{\delta^2}\right)^{\ell} < \infty.\]

This completes the proof of the first claim. For the second claim, 
\[\ex{|U_n|} = \ex{\sum_{i=3}^n X^n_i} > (n-2)\delta^2.\]
\end{proof}

\section{Density Function Derivation}
\label{sec::fnzsjust}
In this section, we derive the density function described in \eqref{eqn::fnzsdef}. Let $X_1 = V[s-1], X_2 = Z^n_{1,2}(s) - \xi^{1,2}(s)$, and $X_3 = \xi^{1,2}(s)$. Then $X_1$ and $X_3$ are absolutely continuous with respective densities $f_1\defeq f^n_{s-1}$ and $f_3\defeq f_{\xi,s}$. Furthermore, $(X_1,X_2)$ is independent of $X_3$. Let $X_4 = X_2 + X_3$. Then the claim is that the joint density of $X_1$ and $X_4$ is given by $f_{14}(x_1,x_4) = f_1(x_1)\ex{f_3(x_4-X_2)|X_1=x_1}$, where $\ex{\cdot|X_1 =x_1}$ is the expectation taken with respect to the regular conditional distribution of $X_2$ given $X_1$ evaluated at $X_1 = x_1$. To verify this, fix any bounded, measurable functions $g_1: \Xmc_1\defeq \mf_{s-1}(\R^d)^2\times\mf_{s-1}([1:q])^2\times\mf_{s-1}([m])\to \R$ and $g_4: \R^d \to \R$. Then,
\begin{align*}
\int_{\Xmc_1\times\R^d} g_1(x_1)g_4(x_4)f_1(x_1)&\ex{f_3(x_4-X_2)|X_1=x_1}\,\ell_{s-1}(dx_1)\,dx_4 \\
&= \int_{\R^d}g_4(x_4)\int_{\Xmc_1} g_1(x_1) \ex{f_3(x_4-X_2)|X_1=x_1}f(x_1)\,\ell_{s-1}(dx_1)\,dx_4\\
&= \int_{\R^d} g_4(x_4)\ex{g_1(X_1)\ex{f_3(x_4-X_2)|X_1}}\,dx_4\\
&=\int_{\R^d} g_4(x_4)\ex{g_1(X_1)f_3(x_4-X_2)}\,dx_4\\
&=\ex{g_1(X_1)\int_{\R^d} g_4(x_4)f_3(x_4-X_2)\,dx_4}\\
&=\ex{g_1(X_1)\int_{\R^d} g_4(x_3+X_2)f_3(x_3)\,dx_3}\\
&=\ex{g_1(X_1)\ex{\int_{\R^d}g_4(x_3+X_2)f_3(x_3)\,dx_3|X_1,X_2}}\\
&=\ex{g_1(X_1)\ex{g_4(X_3+X_2)|X_1,X_2}}\\
&= \ex{g_1(X_1)g_4(X_4)}.
\end{align*}
The second equality above is a consequence of the definition of a regular conditional distribution. The fourth to last equality follows from a simple change of variable of the form $x_3 = x_4-X_2$. The other steps involve Fubini's Theorem and the tower property of conditional expectations.

\ind It follows from standard arguments (e.g. monotone class theorem) that for any bounded, measurable $g: \Xmc_1\times \R^d\to\R$, $\int_{\Xmc_1\times\R^d} g(x_1,x_4)f_{14}(x_1,x_4)\,dx_1\,dx_4 = \ex{g(X_1,X_4)}$ concluding the proof that $f_{14}$ is the density of $(X_1,X_4)$.

\section{Conditional Propagation of Chaos}
\label{sec::condpropchaos}
We review the conditional propagation of chaos result proven in \cite{GanSpiSus25}, which is of independent interest and is the main tool used in the proof of our main results. We begin with a classical descriptor of propagation of chaos:

\begin{proposition}[Propagation of Chaos]
\label{prop::convres}
Let $\{X^n_{1:n}\}_{n \in \N}$ be a triangular array of $\Xmc$-random elements such that for all $n \in \N$, $(X^n_i)_{i=1}^n$ is exchangeable. Suppose $\eta \defeq \law(X)\in \P(\Xmc)$ is a deterministic probability measure. Then,
\[\eta_X \defeq \frac{1}{n}\sum_{i=1}^n \delta_{X^n_i} \to\eta \te{ in probability}\]
if and only if for any $k \in \N$,
\[(X^n_1,\dots,X^n_k) \Rightarrow (X\pcar{1},\dots,X\pcar{k}),\]
where $X\pcar{1},\dots,X\pcar{k}$ are i.i.d. copies of $X$. Moreover, the ``if'' statement of this lemma holds even when $k = 2$.
\end{proposition}
\begin{proof}
This is just a restatement of \cite[Proposition 2.2(i)]{Szn91} in our notation.
\end{proof}

As mentioned in Remark \ref{rmk::munaitdefcpropchaos}, this describes \eqref{eqn::muntdef} of Theorem \ref{thm::hydro}. However, it is not sufficient to establish \eqref{eqn::munaitdef} of the same theorem. In this section, we mention the extension of  Proposition \ref{prop::convres} to a more general case, which was called \emph{conditional propagation of chaos} in \cite{GanSpiSus25}. Conditional propagation of chaos plays a key role in the proof of Theorem \ref{thm::Sampleconv}. We first recall some definitions, also from \cite{GanSpiSus25}.
\begin{definition}[Continuous Dependence]
\label{def::depcont}
Let $X$ and $Y$ be $\Xmc$ and $\Ymc$-random elements respectively. Then we say $Y$ \emph{depends continuously on} $X$ if there exists an $X$-almost surely continuous function $\phi: \Xmc \to \Ymc$ such that $\phi(X) = Y$ almost surely.
\end{definition}

\begin{definition}[Exchangeability Excluding $i$]
\label{def::exchexcli}
For any $n \in \N$ and $i \in [1:n]$, a $\Xmc$-random element $X$ and a collection of $n$ $\Ymc$-random elements $Y \defeq (Y_1,\dots,Y_n)$ are said to be \emph{exchangeable excluding} $i$ if for every permutation $\sigma \in S_n$,
\[\left(X,(Y_j)_{j=1}^n\right) \deq \left(X,(Y_{\sigma(j)})_{j=1}^n\right)\te{ when for } j=i, \sigma(i) = i.\] 
\end{definition}

It is convenient to combine Definitions \ref{def::depcont} and \ref{def::exchexcli} into a single property.

\begin{definition}[$\Xmc/\Ymc$-Convenience]
\label{def::usefulassu}
Let $\mathbf{X} \defeq \{X^n\}_{n \in \N}\cup \{X\}$ be a sequence of $\Xmc$-valued random elements and let $\mathbf{Y} \defeq \{(Y^n_i)_{i\in [1:n]}\}_{n \in \N}\cup\{Y\}$ be a collection of random vectors with entries in $\Ymc$. Then $(\mathbf{X},\mathbf{Y})$ is said to be $\Xmc/\Ymc$-\emph{convenient} if for each $n \in \N$, $(X^n,Y^n)$ is exchangeable excluding 1 and $\eta_{XY} \defeq \law(X,Y|X)$ depends continuously on $X$.
\end{definition}

We now state the conditional propagation of chaos result, which we was introduced and proven in \cite{GanSpiSus25}.

\begin{proposition}[Conditional Propagation of Chaos]
\label{prop::condconvres}
Suppose that $\Xmc$ and $\Ymc$ are locally compact. Let $(\mathbf{X},\mathbf{Y})$ be an $\Xmc/\Ymc$-convenient collection of random elements. Then
\begin{align}
\label{eqn::muXYiconv}
(X^n,\eta^n_{XY})&\defeq \left(X^n,\frac{1}{n}\sum_{j=1}^n \delta_{X^n,Y^n_j}\right) \Rightarrow (X,\eta_{XY}),
\end{align}
if and only if for every $k \in \N\setminus \{1,2\}$,
\begin{equation}
\label{eqn::XXXYYconv}
(X^n,Y^n_{2:k}) \Rightarrow (X,Y\pcar{2:k}),
\end{equation}
where
\begin{itemize}
\item $\law(X,Y\pcar{j}|X) = \eta_{XY}$ for all $j$,
\item and $(Y\pcar{j})_{j=2}^k$ are conditionally independent given $X$.\end{itemize}
In addition, \eqref{eqn::XXXYYconv} implies \eqref{eqn::muXYiconv} when $k = 3$.
\end{proposition}

\begin{definition}[Conditional Propagation of Chaos]
\label{CPC}
We say a pair of random variable collections $(\mathbf{X},\mathbf{Y})$ satisfies the $\Xmc/\Ymc$-\emph{conditional propagation of chaos} property if $(\mathbf{X},\mathbf{Y})$ is $\Xmc/\Ymc$-convenient and satisfies \eqref{eqn::muXYiconv}.
\end{definition}

\bibliographystyle{plain}
\bibliography{reference}
\end{document}